\documentclass[a4paper,11pt]{article}

\usepackage[utf8]{inputenc}
\usepackage[british]{babel}

\usepackage[mono=false]{libertine}
\usepackage[T1]{fontenc}

\usepackage{amsthm}
\usepackage[cmintegrals,libertine]{newtxmath}
\usepackage[cal=dutchcal, calscaled=.95, scr=boondoxo]{mathalfa}
\useosf
\let\mathcal\mathbcal

\usepackage{microtype}

\usepackage{numprint}

\usepackage[explicit]{titlesec}

\titleformat{\section}[block]{\normalfont\large\bfseries\boldmath\centering}{\raggedright\makebox[1em][l]{\thesection.}}{.25em}{#1}
\titleformat{\subsection}[runin]{\normalfont\bfseries\boldmath}{\raggedright\makebox[1em][l]{\thesubsection.}}{1em}{#1.~~\hbox{---}}
\titleformat{\subsubsection}[runin]{\normalfont\itshape}{\raggedright\makebox[1em][l]{\thesubsection.}}{1em}{#1.~~\hbox{---}}

\usepackage[a4paper,vmargin={3cm,3cm},hmargin={2.5cm,2.5cm}]{geometry}
\selectfont

\usepackage{tikz}
	\usetikzlibrary{calc}
	\usetikzlibrary{decorations.markings}

\usepackage[font=sf]{caption}
\usepackage{hyperref}
\hypersetup{
	colorlinks=true,
		citecolor=blue!60!black,
		linkcolor=red!60!black,
		urlcolor=green!40!black,
		filecolor=yellow!50!black,
	breaklinks=true,
	pdfpagemode=UseNone,
	bookmarksopen=false,
}

\usepackage{enumitem}

\newtheorem{thm}{\bfseries \upshape Theorem}
\newtheorem{lem}{Lemma}
\newtheorem{prop}{Proposition}
\newtheorem{cor}{Corollary}

\theoremstyle{definition}

\renewcommand{\le}{\leqslant}
\renewcommand{\ge}{\geqslant}

\newcommand{\ensembles}[1]{\mathbf{#1}}
	\newcommand{\N}{\ensembles{N}}
	\newcommand{\Z}{\ensembles{Z}}
	\newcommand{\R}{\ensembles{R}}
	
	\renewcommand{\P}{\ensembles{P}}
	\newcommand{\E}{\ensembles{E}}
	\newcommand{\Var}{\mathrm{Var}}

	\newcommand{\1}{\ensembles{1}}
\newcommand{\ind}[1]{\1_{\{#1\}}}

\renewcommand{\Pr}[1]{\P\left(#1\right)}
\newcommand{\Prc}[2]{\P\left(#1 \;\middle|\; #2\right)}
\newcommand{\Es}[1]{\E\left[#1\right]}
\newcommand{\Esc}[2]{\E\left[#1 \;\middle|\; #2\right]}

\newcommand{\Map}{\mathbf{M}}
\newcommand{\PMap}{\mathbf{PM}}

\newcommand{\Tree}{\mathbf{T}}
\newcommand{\LTree}{\mathbf{LT}}

\newcommand{\degr}{d}

\newcommand{\Fn}{\Tree_{\degr_n}^{\varrho_n}}

\newcommand{\fn}{T_{\degr_n}^{\varrho_n}}
\newcommand{\tn}{T_{\degr_n}}

\newcommand{\Wfn}{W_{\degr_n}^{\varrho_n}}

\newcommand{\Bfn}{B_{\degr_n}^{\varrho_n}}
\newcommand{\Btn}{B_{\degr_n}}

\newcommand{\Lfn}{L_{\degr_n}^{\varrho_n}}

\newcommand{\LFn}{\LTree_{\degr_n}^{\varrho_n}}

\newcommand{\Mn}{\Map_{\degr_n}^{\varrho_n}}
\newcommand{\mn}{M_{\degr_n}^{\varrho_n}}
\newcommand{\Vmn}{V(M_{\degr_n}^{\varrho_n})}
\newcommand{\PMn}{\PMap_{\degr_n}^{\varrho_n}}

\newcommand{\dgr}{d_{\mathrm{gr}}}
\newcommand{\pgr}{p_{\mathrm{unif}}}

\newcommand{\h}{\mathrm{ht}}
\newcommand{\wid}{\mathrm{wid}}

\renewcommand{\d}{\mathrm{d}}
\newcommand{\e}{\mathrm{e}}

\newcommand{\m}{\mathbf{m}}

\newcommand{\Cont}{\mathsf{Cont}}
\newcommand{\LR}{\mathsf{LR}}
\newcommand{\RR}{\mathsf{R}}
\newcommand{\LL}{\mathsf{L}}

\newcommand{\cv}[1][n]{\enskip\mathop{\longrightarrow}^{}_{#1 \to \infty}\enskip}
\newcommand{\cvloi}[1][n]{\enskip\mathop{\longrightarrow}^{(d)}_{#1 \to \infty}\enskip}

\newcommand{\cvproba}[1][n]{\enskip\mathop{\longrightarrow}^{\P}_{#1 \to \infty}\enskip}

\title{On the growth of random planar maps with a prescribed degree sequence}

\author{Cyril \textsc{Marzouk}
\thanks{CNRS, IRIF UMR 8243, Universit\'{e} Paris-Diderot, France.\hfill  \href{mailto:cmarzouk@irif.fr}{\texttt{cmarzouk@irif.fr}}
\newline This work was supported by first a public grant as part of the Fondation Mathématique Jacques Hadamard and then the European Research Council, grant \texttt{ERC-2016-STG 716083} (CombiTop).}}

\begin{document}

\maketitle

\begin{abstract}\linespread{1}\selectfont
For non-negative integers $(\degr_n(k))_{k \ge 1}$ such that $\sum_{k \ge 1} \degr_n(k) = n$, we sample a bipartite planar map with $n$ faces uniformly at random amongst those which have $\degr_n(k)$ faces of degree $2k$ for every $k \ge 1$ and we study its asymptotic behaviour as $n \to \infty$. We prove that the diameter of such maps grow like $\sigma_n^{1/2}$, where $\sigma_n^2 = \sum_{k \ge 1} k (k-1) \degr_n(k)$ is a global variance term. More precisely, we prove that the vertex-set of these maps equipped with the graph distance divided by $\sigma_n^{1/2}$ and the uniform probability measure always admits subsequential limits in the Gromov--Hausdorff--Prokhorov topology.

Our proof relies on a bijection with random labelled trees; we are able to prove that the label process is always tight when suitably rescaled, even if the underlying tree is not tight for the Gromov--Hausdorff topology. We also rely on a new spinal decomposition which is of independent interest. Finally this paper also serves as a toolbox for a companion paper in which we discuss more precisely Brownian limits of such maps.
\end{abstract}

\section{Introduction}

This paper deals with scaling limits of random planar maps as their size tends to infinity and the edge-lengths tend to zero. We consider quite general distributions on maps which are essentially configuration models on their dual maps studied previously in~\cite{Marzouk:Scaling_limits_of_random_bipartite_planar_maps_with_a_prescribed_degree_sequence}. Let us first present precisely the model.

\subsection{Model and notation}
Recall that a (rooted planar) map is a finite (multi-)graph embedded in the two-dimensional sphere, in which one oriented edge is distinguished (the \emph{root-edge}), and viewed up to orientation-preserving homeomorphisms to make it a discrete object. The embedding allows to define the \emph{faces} of the map which are the connected components of the complement of the graph on the sphere, and the \emph{degree} of a face is the number of edges incident to it, counted with multiplicity: an edge incident on both sides to the same face contributes twice to its degree. For technical reasons, we restrict ourselves to \emph{bipartite} maps, in which all faces have even degree. The face incident to the right of the root-edge is called the \emph{root-face}, whereas the other faces are called \emph{inner faces}. The collection of edges incident to the root-face is the \emph{boundary} of the map. A map with only two boundary edges can be seen as a map without boundary by gluing these two edges together.

For every integer $n \ge 1$, we are given an integer $\varrho_n \ge 1$ and a sequence $(\degr_n(k))_{k \ge 1}$ in $\Z_+^\N$ such that $\sum_{k \ge 1} \degr_n(k) = n$, and we let $\Mn$ denote the set of of all those bipartite maps with boundary-length $2 \varrho_n$ and $n$ inner faces, amongst which exactly $k$ have degree $2k$ for every $k \ge 1$; see Figure~\ref{fig:exemple_carte} for an example. One can easily check that any such map has $\upsilon_n \coloneqq \varrho_n + \sum_{k \ge 1} k \degr_n(k)$ edges, and so $\degr_n(0) + 1 \coloneqq \upsilon_n - n + 1$ vertices by Euler's formula. Then for every $n \ge 1$, the set $\Mn$ is finite and non-empty, and its cardinal is precisely given by
\[\#\Mn = 
\frac{2 \varrho_n (\upsilon_n-1)!}{(\degr_n(0)+1)!} 
\binom{2\varrho_n-1}{\varrho_n-1}
\prod_{k \ge 1} \frac{1}{\degr_n(k)!} \binom{2k-1}{k-1}^{\degr_n(k)}.\]
This can be derived from the `slicing formula' of Tutte~\cite{Tutte:A_census_of_slicings}, but it is also an easy consequence of the bijective method we shall use which relates planar maps and \emph{labelled forests}.
A key quantity in this work is
\[\sigma_n^2 \coloneqq \sum_{k \ge 1} k (k-1) \degr_n(k),\]
which is a sort of global variance term. We sample $\mn$ uniformly at random in $\Mn$ and consider its asymptotic behaviour as $n \to \infty$. A particular, well-studied case of such random maps are so-called random \emph{quadrangulations} and more generally \emph{$2p$-angulations}, with any $p \ge 2$ fixed, where all faces have degree $2p$, which corresponds to our model in which $\degr_n(p) = n$ so $\sigma_n^2 = p (p-1) n$. Another well-studied model is that of \emph{Boltzmann planar maps} introduced by Marckert and Miermont~\cite{Marckert-Miermont:Invariance_principles_for_random_bipartite_planar_maps}, which can be seen as a mixture of our model in which the face-degrees $\degr_n$ are random and then one samples $\mn$ conditionally given $\degr_n$.

\begin{figure}[!ht] \centering
\includegraphics[height=8\baselineskip, page = 14]{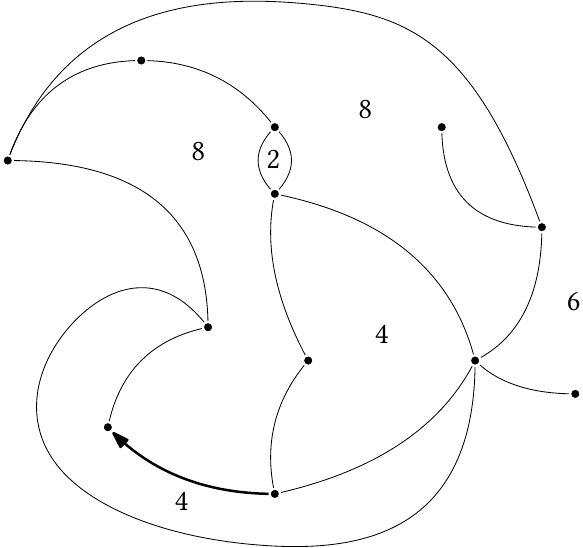}
\caption{An element of $\Mn$ with $n = 5$ inner faces, $\varrho_n = 4$ half boundary-length and $(\degr_n)_{n \ge 1} = (1, 2, 1, 1, 0, 0, \dots)$ face-degree sequence; faces are labelled by their degree.}
\label{fig:exemple_carte}
\end{figure}

\subsection{Main result}
Roughly speaking, a map can be viewed as the topological gluing of polygons which forms a sphere; we are given $n+1$ polygons whose half-number of sides is prescribed by $\degr_n$ and $\varrho_n$, and we choose such a gluing uniformly at random. The simplest such model is $\varrho_n = 1$ and $\degr_n(2)=n$ when all polygons are quadrangles, in which case it was first shown by Chassaing \& Schaeffer~\cite{Chassaing-Schaeffer:Random_planar_lattices_and_integrated_super_Brownian_excursion} that the diameter of a uniformly random quadrangulation with $n$ faces scales like $n^{1/4}$ and converges in distribution after this scaling. Later, Le Gall~\cite{Le_Gall:The_topological_structure_of_scaling_limits_of_large_planar_maps} proved that the vertex-set of such random quadrangulations endowed with the graph distance multiplied by $n^{-1/4}$ admits subsequential limits in the sense of the so-called \emph{Gromov--Hausdorff} topology, and this was extended by Bettinelli~\cite{Bettinelli:Scaling_limit_of_random_planar_quadrangulations_with_a_boundary} to quadrangulations with a boundary, when $\varrho_n$ is arbitrary. We prove that tightness is very general. In addition to the graph distance $\dgr$ on the vertex-set $\Vmn$ of the map $\mn$, we also include the uniform probability measure $\pgr$.

\begin{thm}
\label{thm:tension_cartes}
Fix any sequence of boundary-lengths $(\varrho_n)_{n \ge 1}$ and any degree sequence $(\degr_n)_{n \ge 1}$. From every increasing sequence of integers, one can extract a subsequence along which the sequence of metric measured spaces
\[\left(\Vmn, (\sigma_n + \varrho_n)^{-1/2} \dgr, \pgr\right)_{n \ge 1}\]
converges in distribution in the sense of Gromov--Hausdorff--Prokhorov.
\end{thm}

Considering a map with a boundary seems artificial since the result does not depend on the root-edge: we could add the face of degree $2\varrho_n$ to the degree sequence $\degr_n$ and replace the factor $(\sigma_n + \varrho_n)^{-1/2}$ by $(\sigma_n')^{-1/2}$, where $(\sigma_n')^2 = \sigma_n^2 + \varrho_n (\varrho_n - 1)$. We chose to highlight the boundary in order to fit in the existing literature.

Let us notice that, although we shall not prove it here, this statement is not void in the sense that we did not choose a scaling factor so large that the spaces degenerate to a single point. Indeed, it is easily shown that a face of degree, say, $d$, has diameter of order $d^{1/2}$ so we see that the diameter of the rescaled map is bounded away from zero whenever there exists a face (the outer face or an inner face) with degree of order $\sigma_n$. If it is not the case, we shall prove in the companion paper~\cite{Marzouk:Scaling_limits_of_planar_trees_and_maps_with_a_prescribed_degree_sequence_Brownian_limits} that these maps converge towards a universal non-degenerate limit called the \emph{Brownian map}~\cite{Marckert-Mokkadem:Limit_of_normalized_random_quadrangulations_the_Brownian_map, Le_Gall:Uniqueness_and_universality_of_the_Brownian_map, Miermont:The_Brownian_map_is_the_scaling_limit_of_uniform_random_plane_quadrangulations}. 

In the paper~\cite{Le_Gall:The_topological_structure_of_scaling_limits_of_large_planar_maps}, Le Gall proves tightness not only of random quadrangulations, but more generally of random \emph{$2p$-angulations}, with any $p \ge 2$ fixed, where all faces have degree $2p$, which corresponds to our model in which $\degr_n(p) = n$ so $\sigma_n^2 = p (p-1) n$. The preceding theorem enables us to let $p$ depend on $n$: fix any sequence $(p_n)_{n \ge 1} \in \{2, 3, \dots\}^\N$ and let $M_{n,p_n}$ be a uniformly chosen random $2p_n$-angulation with $n$ faces, then from every increasing sequence of integers, one can extract a subsequence along which the sequence of metric measured spaces
\[\left(V(M_{n,p_n}), (p_n (p_n-1) n)^{-1/4} \dgr, \pgr\right)_{n \ge 1}\]
converges in distribution in the sense of Gromov--Hausdorff--Prokhorov. 
Le Gall~\cite{Le_Gall:Uniqueness_and_universality_of_the_Brownian_map} and Miermont~\cite{Miermont:The_Brownian_map_is_the_scaling_limit_of_uniform_random_plane_quadrangulations} proved that, for random $2p$-angulations, the subsequential limits agree and these maps converge to the aforementioned Brownian map. More recently Bettinelli and Miermont~\cite{Bettinelli-Miermont:Compact_Brownian_surfaces_I_Brownian_disks} extended this convergence to maps with a boundary, in which case the limit is called a \emph{Brownian disk}. In the companion paper~\cite{Marzouk:Scaling_limits_of_planar_trees_and_maps_with_a_prescribed_degree_sequence_Brownian_limits} we shall obtain the convergence of our general model of random maps towards such limits (as well as the \emph{Brownian Continuum Random Tree} of Aldous~\cite{Aldous:The_continuum_random_tree_3} in a so-called `condensation' phenomenon) under appropriate assumption on the degree sequence. As a particular case, we shall see there that $2p_n$-angulations always converge to the Brownian map.

\subsection{Strategy of the proof and further discussion}
We study planar maps via a bijection with \emph{labelled} forests obtained by combining the works~\cite{Bouttier-Di_Francesco-Guitter:Planar_maps_as_labeled_mobiles, Janson-Stefansson:Scaling_limits_of_random_planar_maps_with_a_unique_large_face} where each vertex in the forest is assigned a random number in $\Z$ in a particular way (see Section~\ref{sec:arbres_etiquetes}). The graph distance on the map can be partially encoded by the \emph{label process} which encodes the labels of the forest, and we prove that this process is tight when suitably rescaled, see Theorem~\ref{thm:tension_etiquettes}. 
This tightness relies only on the so-called \emph{{\L}ukasiewicz path} of the forest, which is a very simple process: up to a discrete Vervaat's transform (also known as cyclic lemma), the latter is a uniformly random path which makes $\degr_n(k)$ jumps of size $k-1$ for each $k \ge 0$. This is the simplest example of a random path with exchangeable increments and its asymptotic behaviour is well-understood.
We stress that tightness of the label process does not require the contour of height process to be tight. As a matter of fact, the latter are not tight in full generality! Going from tightness of the label process to tightness of the map as in Theorem~\ref{thm:tension_cartes} is then very standard in the theory.

Theorem~\ref{thm:tension_cartes} does not say anything about the subsequential limits. In the companion paper~\cite{Marzouk:Scaling_limits_of_planar_trees_and_maps_with_a_prescribed_degree_sequence_Brownian_limits} we shall focus on Brownian limits such maps: we find optimal assumptions on the degrees for the rescaled maps to converge towards the Brownian map, the Brownian disks, and even towards the Brownian CRT. Thanks to the results from this present work, it `only' suffices to characterise the subsequential limits via the convergence of the finite-dimensional marginals of the label process.

In Section~\ref{sec:preliminaires}, we first recall the definition of plane forests and their encoding by paths, then we gather a few results about the {\L}ukasiewicz path of the random forests associated with our random maps. Then in Section~\ref{sec:epine}, we prove a new \emph{spinal decomposition} which describes the behaviour of the ancestral lines of randomly chosen vertices; this result shall be used in the proof of tightness of the label process and it is one of the key ingredients to study its finite-dimensional marginals in~\cite{Marzouk:Scaling_limits_of_planar_trees_and_maps_with_a_prescribed_degree_sequence_Brownian_limits}. Finally, in Section~\ref{sec:arbres_etiquetes}, we recall the bijection between planar maps and labelled forests, we then state and prove tightness of the label process in Theorem~\ref{thm:tension_etiquettes}, before deducing Theorem~\ref{thm:tension_cartes}.

\subsection*{Acknowledgement}

Many, many, thanks to Igor Kortchemski who spotted an important error in a first draft.

\section{Preliminaries on forests and {\L}ukasiewicz paths}
\label{sec:preliminaires}

\subsection{Plane forests and discrete paths}
\label{sec:def_arbres}

We view (plane) trees as words: let $\N = \{1, 2, \dots\}$ and set $\N^0 = \{\varnothing\}$, then a \emph{tree} is a non-empty subset $T \subset \bigcup_{n \ge 0} \N^n$ such that:
$\varnothing \in T$, it is called the \emph{root} of $T$, and for every $x = (x_1, \dots, x_n) \in T$, we have $pr(x) \coloneqq (x_1, \dots, x_{n-1}) \in T$ and there exists an integer $k_x \ge 0$ such that $x i \coloneqq (x_1, \dots, x_n, i) \in T$ if and only if $1 \le i \le k_x$.
We interpret the vertices of a tree $T$ as individuals in a population. For every $x = (x_1, \dots, x_n) \in T$, the vertex $pr(x)$ is its \emph{parent}, $k_x$ is the number of \emph{children} of $x$ (if $k_x = 0$, then $x$ is called a \emph{leaf}, otherwise, $x$ is called an \emph{internal vertex}), and $|x| = n$ is its \emph{generation}; finally, we let $\chi_x \in \{1, \dots, k_{pr(x)}\}$ be the only index such that $x=pr(x)\chi_x$, which is the relative position of $x$ amongst its siblings. We shall denote by $\llbracket x , y \rrbracket$ the unique geodesic path between $x$ and $y$.

Fix a tree $T$ with $\upsilon_n+1$ vertices, listed $\varnothing = x_0 < x_1 < \dots < x_{\upsilon_n}$ in lexicographical order. It is well-known that $T$ is described by each of the following two discrete paths. First, its \emph{{\L}ukasiewicz path} $W = (W(j) ; 0 \le j \le \upsilon_n + 1)$ is defined by $W(0) = 0$ and for every $0 \le j \le \upsilon_n$,
\[W(j+1) = W(j) + k_{x_j}-1.\]
One easily checks that $W(j) \ge 0$ for every $0 \le j \le \upsilon_n$ but $W(\upsilon_n + 1)=-1$. Next, the \emph{height process} $H = (H(j); 0 \le j \le \upsilon_n)$ is defined by setting for every $0 \le j \le \upsilon_n$,
\[H(j) = |x_j|.\]

Without further notice, throughout this work, every {\L}ukasiewicz path shall be viewed as a step function, jumping at integer times, whereas the height processes shall be viewed as continuous functions after interpolating linearly between integer times.

\begin{figure}[!ht]\centering
\def\r{.6}
\def\longueur{2.8}
\begin{tikzpicture}[scale=.4]
\coordinate (1) at (0,0*\longueur);
	\coordinate (2) at (-4.25,1*\longueur);
	\coordinate (3) at (-1.5,1*\longueur);
	\coordinate (4) at (1.5,1*\longueur);
		\coordinate (5) at (.75,2*\longueur);
			\coordinate (6) at (.75,3*\longueur);
				\coordinate (7) at (-.75,4*\longueur);
				\coordinate (8) at (.25,4*\longueur);
				\coordinate (9) at (1.25,4*\longueur);
				\coordinate (10) at (2.25,4*\longueur);
		\coordinate (11) at (2.25,2*\longueur);
	\coordinate (12) at (4.25,1*\longueur);
		\coordinate (13) at (3.5,2*\longueur);
			\coordinate (14) at (2.5,3*\longueur);
			\coordinate (15) at (3.5,3*\longueur);
			\coordinate (16) at (4.5,3*\longueur);
		\coordinate (17) at (5,2*\longueur);

\draw[dashed]	(1) -- (2)	(1) -- (3)	(1) -- (4)	(1) -- (12);
\draw
	(4) -- (5)	(4) -- (11)
	(5) -- (6)
	(6) -- (7)	(6) -- (8)	(6) -- (9)	(6) -- (10)
	(12) -- (13)	(12) -- (17)
	(13) -- (14)	(13) -- (15)	(13) -- (16)
;

\foreach \x in {1, 2, ..., 17}
\draw[fill=black] (\x) circle (3pt);

\begin{footnotesize}
\draw (1) node[below] {0};
\draw (2) node[left] {1};
\draw (3) node[left] {2};
\draw (4) node[left] {3};
\draw (5) node[left] {4};
\draw (6) node[left] {5};
\draw (7) node[above] {6};
\draw (8) node[above] {7};
\draw (9) node[above] {8};
\draw (10) node[above] {9};
\draw (11) node[left] {10};
\draw (12) node[right] {11};
\draw (13) node[right] {12};
\draw (14) node[above] {13};
\draw (15) node[above] {14};
\draw (16) node[above] {15};
\draw (17) node[right] {16};
\end{footnotesize}
\begin{footnotesize}
\begin{scope}[shift={(10,7)}]
\draw[thin, ->]	(0,3) -- (17,3);
\draw[thin, ->]	(0,-1) -- (0,5.5);
\foreach \x in {-1, 0, 1, 2, 4, 5}
	\draw[dotted]	(0,\x) -- (17,\x);
\foreach \x in {-1, 0, 1, ..., 5}
	\draw (.1,\x)--(-.1,\x);
\draw
	(0,-1) node[left] {-4}
	(0,0) node[left] {-3}
	(0,1) node[left] {-2}
	(0,2) node[left] {-1}
	(0,3) node[left] {0}
	(0,4) node[left] {1}
	(0,5) node[left] {2}
;
\foreach \x in {2, 4, ..., 16}
	\draw (\x,3.1)--(\x,2.9)	(\x,3) node[below] {\x};

\draw[very thick]
	(0,3) -- ++ (2,0)
	++(0,-1) -- ++ (1,0)
	++(0,-1) -- ++ (1,0)
	++(0,1) -- ++ (1,0)
	++(0,0) -- ++ (1,0)
	++(0,3) -- ++ (1,0)
	++(0,-1) -- ++ (1,0)
	++(0,-1) -- ++ (1,0)
	++(0,-1) -- ++ (1,0)
	++(0,-1) -- ++ (1,0)
	++(0,-1) -- ++ (1,0)
	++(0,1) -- ++ (1,0)
	++(0,2) -- ++ (1,0)
	++(0,-1) -- ++ (1,0)
	++(0,-1) -- ++ (1,0)
	++(0,-1) -- ++ (1,0)
;
\draw[fill=black] (17,-1) circle (2pt);
\end{scope}
\begin{scope}[shift={(10,0)}]
\draw[thin, ->]	(0,0) -- (17,0);
\draw[thin, ->]	(0,0) -- (0,4.5);
\foreach \x in {1, 2, ..., 4}
	\draw[dotted]	(0,\x) -- (17,\x);
\foreach \x in {0, 1, ..., 4}
	\draw (.1,\x)--(-.1,\x)	(0,\x) node[left] {\x};
\foreach \x in {2, 4, ..., 16}
	\draw (\x,.1)--(\x,-.1)	(\x,0) node[below] {\x};

\draw[fill=black]
	(0, 0) circle (2pt) --
	++ (1, 1) circle (2pt) --
	++ (1, 0) circle (2pt) --
	++ (1, 0) circle (2pt) --
	++ (1, 1) circle (2pt) --
	++ (1, 1) circle (2pt) --
	++ (1, 1) circle (2pt) --
	++ (1, 0) circle (2pt) --
	++ (1, 0) circle (2pt) --
	++ (1, 0) circle (2pt) --
	++ (1, -2) circle (2pt) --
	++ (1, -1) circle (2pt) --
	++ (1, 1) circle (2pt) --
	++ (1, 1) circle (2pt) --
	++ (1, 0) circle (2pt) --
	++ (1, 0) circle (2pt) --
	++ (1, -1) circle (2pt)
;
\end{scope}
\end{footnotesize}
\end{tikzpicture}
\caption{A planted forest on the left with the lexicographical order of the vertices, and on the right,  its {\L}ukasiewicz path on top and its height process below.}
\label{fig:codage_foret}
\end{figure}
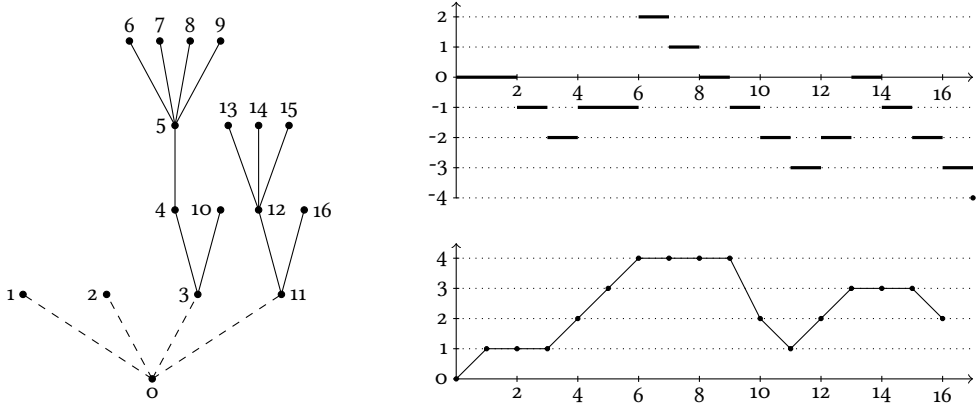

A (plane) \emph{forest} is a finite ordered list of plane trees, which we may view as a single tree by linking all the roots to an extra root-vertex; to be consistent, we also perform this operation when the forest originally consists of a single tree. Then we may define the {\L}ukasiewicz path, height and contour processes of a forest as the paths describing the corresponding tree. 
In this case, the first jump of the {\L}ukasiewicz path is given by the number of trees minus one, say, $\varrho-1$, and then the path terminates by hitting $-1$ for the first time. This first possibly large jump is uncomfortable and we prefer instead cancelling it so the {\L}ukasiewicz path starts at $0$, makes a first $0$ step, and it terminates by hitting $-\varrho$ for the first time. We refer to Figure~\ref{fig:codage_foret} for an illustration.

The next lemma, whose proof is left as an exercise, gathers some deterministic results that we shall need (we refer to e.g. Le Gall~\cite{Le_Gall:Random_trees_and_applications} for a thorough discussion of such results). In order to simplify the notation, we shall identify the vertices of a tree with their index in the lexicographic order: if $x$ is the $i$-th vertex of a tree $T$ whose {\L}ukasiewicz path is $W$, then we write $W(x)$ for $W(i)$.

\begin{lem}\label{lem:codage_marche_Luka}
Let $T$ be a plane tree and $W$ be its {\L}ukasiewicz path. Fix a vertex $x \in T$, then
\[W(x k_x) = W(x),
\qquad
W(xj') = \inf_{[xj, xj']} W,
\qquad\text{and}\qquad
j' - j = W(xj) - W(xj')\]
for every $1 \le j \le j' \le k_x$.
\end{lem}

For any vertex $x$ of a tree $T$, let $\RR(x)$ be the number those vertices $y > x$ whose parent is a strict ancestor of $x$, and let $\LL(x)$ be the number those vertices $y < x$ whose parent is a strict ancestor of $x$ and which themselves are not. In other words, $\RR(x)$ and $\LL(x)$ count respectively the number of vertices branching-off of the ancestral line $\llbracket \varnothing, x\llbracket$, those on the left for $\LL(x)$ and those on the right for $\RR(x)$ (and not $x$ itself). We put $\LR(x) = \LL(x) + \RR(x)$. Then a consequence of this lemma is that, in a tree, we have $\RR(x) = W(x)$. In the case of a forest, we define $\LL(x)$ and $\RR(x)$ (and so $\LR(x)$) as the same quantities \emph{in the tree} containing $x$, i.e. we really do want to view the forest as a forest and not as a tree, and then we have more generally
\[\RR(x) = W(x) - \min_{0 \le y \le x} W(y).\]
For example, if $x$ is the vertex number 8 in Figure~\ref{fig:codage_foret}, then $\RR(x) = W(8) - \min_{0 \le i \le 8} W(i) = 2$, which corresponds to the vertices 9 and 10: we do not count the 11-th one.

\subsection{The {\L}ukasiewicz path of a random forest as a cyclicly shifted bridge}
\label{sec:Lukasiewicz}

We consider random forest built as follows. Recall the notation from introduction; for every integer $n$, we consider the degree sequence $\degr_n = (\degr_n(k))_{k \ge 0} \in \Z_+^{\Z_+}$ and let $\Tree_{\degr_n}^{\varrho_n}$ be the set of all those plane forests with $\varrho_n = \sum_{k \ge 0} (1-k) \degr_n(k)$ trees and $\upsilon_n = \sum_{k \ge 0} \degr_n(k)$ vertices, amongst which $\degr_n(k)$ have $k$ children for every $k \ge 0$; for a single tree, we simply write $\Tree_{\degr_n}$ for $\Tree_{\degr_n}^{1}$. It is well-known (see e.g.~\cite[Chapter~6.2]{Pitman:Combinatorial_stochastic_processes}) that the cardinal of this set is given by 
\[\#\Tree_{\degr_n}^{\varrho_n} 
= \frac{\varrho_n}{\upsilon_n} \binom{\upsilon_n}{(\degr_n(k))_{k \ge 0}}
= \frac{\varrho_n (\upsilon_n-1)!}{\prod_{k \ge 0} \degr_n(k)!}.\]
We sample $\fn$ uniformly at random in $\Tree_{\degr_n}^{\varrho_n}$ and consider its asymptotic behaviour as $n \to \infty$. Recall the notation $\sigma_n^2 = \sum_{k \ge 1} k (k-1) \degr_n(k)$ which is sort of global variance of the offspring distribution. Let us also denote by
\[\Delta_n \coloneqq \max\{k \ge 0 : \degr_n(k) > 0\}\]
the largest offspring of the forest and by
\[\epsilon_n \coloneqq \sum_{k \ge 1} k \degr_n(k) = \upsilon_n - \varrho_n\]
the number of edges of the forest (viewed as a forest).

The key to understand the behaviour of the {\L}ukasiewicz path $\Wfn$ of our random forest $\fn$ is the well-known representation as a cyclic shift of a \emph{bridge}, also called the discrete Vervaat's transform, see e.g. Pitman~\cite[Chapter~6]{Pitman:Combinatorial_stochastic_processes} for details. Fix two integers $m, \varrho \ge 1$ and a vector $b = (b_1, \dots, b_m) \in \{-1, 0, 1, 2, \dots\}^m$ such that $b_1 + \dots + b_m = -\varrho$. For any $i, j \in \{1, \dots, m\}$, we put $b^{i}_j = b_{i+j}$, where the sum is understood modulo $m$. The sequence $b^{i}$ is called the $i$-th cyclic shift of $b$. Clearly, we have $b^{i}_1 + \dots + b^{i}_m = -\varrho$ for every $i$. Let $I = \inf_{1 \le \ell \le m} b_1 + \dots + b_\ell$,
then the following properties are equivalent:
\begin{enumerate}
\item $b^{i}_1 + \dots + b^{i}_k \ge 1 - \varrho$ for every $1 \le k \le m-1$,

\item there exists $p \in \{0, \dots, \varrho - 1\}$ such that $i = \inf\{k \in\{1, \dots, m\} : b_1 + \dots + b_k = I+p\}$.
\end{enumerate}
In words, the path obtained by summing successively the $b_j$'s is a bridge started from $0$ and ending at $-\varrho$, and there is exactly $\varrho$ ways to cyclicly shift it to turn it to a \emph{first-passage bridge}. In the case $\varrho = 1$, we have $p = 0$ so $i$ is simply the left-most argmin of the bridge.

Let $\Bfn = (\Bfn(k))_{0 \le k \le \upsilon_n}$ be a discrete path sampled uniformly at random amongst all those started from $\Bfn(0) = 0$ and which make exactly $\degr_n(\ell)$ jumps with value $\ell-1$ for every $\ell \ge 0$, so $\Bfn$ is a bridge from $0$ to $\Bfn(\upsilon_n) = \sum_{\ell \ge 0} (\ell - 1) \degr_n(\ell) = - \varrho_n$. Independently, sample $p_n$ uniformly at random in $\{0, \dots, \varrho_n-1\}$ and set 
\[i_n = \inf\left\{k \in\{1, \dots, \upsilon_n\} : \Bfn(k) = p_n + \inf_{1 \le \ell \le \upsilon_n} \Bfn(\ell)\right\}.\]
Then according to the preceding discussion, $\Wfn$ has the law of $(\Bfn(i_n+k) - \Bfn(i_n))_{1 \le k \le \upsilon_n}$, where again the sum is understood modulo $\upsilon_n$. Moreover, one can check that $i_n$ has the uniform distribution on $\{1, \dots, \upsilon_n\}$ and is independent of the cyclicly shifted path.

The fact that all the jumps of the random bridge $\Bfn$ are larger than or equal to $-1$ leads to exponential decay of the negative tails. Indeed, the arguments developed by Addario-Berry~\cite[Section~3]{Addario_Berry:Tail_bounds_for_the_height_and_width_of_a_random_tree_with_a_given_degree_sequence} in the case $\varrho_n = 1$ are easily extended to show that the `remaining sequence' (remaining space divided by the remaining time)
\[\left(\frac{- \varrho_n - \Bfn(i)}{\upsilon_n-i}\right)_{0 \le i \le \upsilon_n - 1}\]
is a martingale for the natural filtration $(F_i)_{0 \le i \le \upsilon_n - 1}$, started from $-\varrho_n / \upsilon_n$, which satisfies furthermore that for every $0 \le i \le \upsilon_n - 2$,
\[- \frac{\Bfn(i+1) + \varrho_n}{\upsilon_n-(i+1)} + \frac{\Bfn(i) + \varrho_n}{\upsilon_n-i} \le \frac{\upsilon_n}{(\upsilon_n-(i+1))^2},\]
and
\[\Var\left(\frac{\Bfn(i+1) + \varrho_n}{\upsilon_n-(i+1)} \;\middle|\; F_i\right) \le \frac{\sigma_n^2}{(\upsilon_n - (i+1))^3}.\]
Then the martingale Chernoff-type bound, see e.g. McDiarmid~\cite{McDiarmid:Concentration}, Theorem~3.15 and the remark at the end of Section~3.5 there, claims that for every $z \ge 0$ and every $1 \le k \le \upsilon_n-1$,
\[\Pr{\min_{i \le k} \frac{\Bfn(i) + \varrho_n}{\upsilon_n-i} - \frac{\varrho_n}{\upsilon_n} \le - z}
\le \exp\left(- \frac{z^2}{2 \frac{k \sigma_n^2}{(\upsilon_n - k)^3} + \frac{2}{3} z \frac{\upsilon_n}{(\upsilon_n - k)^2}}\right)
\le \exp\left(- \frac{(\upsilon_n - k)^3 z^2}{2k \sigma_n^2 + \upsilon_n^2 z}\right).\]
Observe that $\frac{\Bfn(i) + \varrho_n}{\upsilon_n-i} - \frac{\varrho_n}{\upsilon_n} = \frac{\Bfn(i) + i \varrho_n / \upsilon_n}{\upsilon_n-i}$ so for every $z \ge 0$ and every $1 \le k \le \upsilon_n-1$,
\[\Pr{\min_{i \le k} \left(\Bfn(i) + \frac{i \varrho_n}{\upsilon_n}\right) \le - z}
\le \Pr{\min_{i \le k} \frac{\Bfn(i) + i \varrho_n / \upsilon_n}{\upsilon_n-i} \le - \frac{z}{\upsilon_n}}
\le \exp\left(- \frac{(\upsilon_n - k)^3 z^2 / \upsilon_n^2}{2k \sigma_n^2 + \upsilon_n z}\right).\]
Finally, since $\Bfn(i) + i \varrho_n / \upsilon_n \le \Bfn(i) + k \varrho_n / \upsilon_n$ when $i \le k$, 
taking $k = \alpha \upsilon_n$ with $\alpha \in (0,1)$, we have for every $z \ge 0$,
\begin{equation}\label{eq:queue_exp_min_pont_bis}
\Pr{\min_{i \le \alpha \upsilon_n} \Bfn(i) + \alpha \varrho_n \le - \sigma_n z}
\le \exp\left(- \frac{(\upsilon_n - \alpha \upsilon_n)^3 (\sigma_n z)^2 / \upsilon_n^2}{2 \alpha \upsilon_n \sigma_n^2 + \upsilon_n \sigma_n z}\right)
\le \exp\left(- \frac{(1 - \alpha)^3 \sigma_n z^2}{2 \alpha \sigma_n + z}\right).
\end{equation}
In the next two subsections we apply this tail bound to derive precious information about the {\L}ukasiewicz path.

\subsection{Further results}
\label{sec:foret_sans_degre_un}

When dealing with forests, vertices with out-degree $1$ play an important role and removing them may drastically change the geometry. On the other hand, when dealing with planar maps, these correspond to faces of degree $2$ which play no role in the geometry: precisely, if one glues together the two edges of every face of degree $2$, then one does not affect the number of vertices nor their graph distance, so the metric measured space induced by the map stays unchanged. The point is that a forest in which the proportion of vertices with out-degree $1$ is bounded is easier to study. For example, assume that $\degr_n(1) \le (1-\delta) \epsilon_n$ for some $\delta \in (0,1)$, then we have
\[\sigma_n^2 
= \sum_{k \ge 2} k (k-1) \degr_n(k) 
\ge \sum_{k \ge 2} k \degr_n(k) 
= \epsilon_n - \degr_n(1)
\ge \delta \epsilon_n,\]
and similarly
\[\degr_n(0)
= \varrho_n + \sum_{k \ge 2} (k-1) \degr_n(k) 
\ge \varrho_n + \sum_{k \ge 2} \frac{k}{2} \degr_n(k) 
\ge \varrho_n + \frac{\delta}{2} \epsilon_n 
\ge \frac{\delta}{2} \upsilon_n.\]

\begin{prop}\label{prop:moments_marche_Luka}
Fix $\delta \in (0,1)$. For every $n \in \N$, every degree sequence $\degr_n$ such that $\degr_n(1) \le (1 - \delta) \epsilon_n$, every $0 \le s < t \le 1$ with $|t - s| \le 1/2$, and every $x > 0$, it holds that
\[\Pr{\Wfn(\upsilon_ns) - \min_{s \le r \le t} \Wfn(\upsilon_nr) - \varrho_n |t-s|> \sigma_n |t-s|^{1/2} x} \le 2 \exp\left(- \frac{x^2}{16 (2 + \delta^{-1})}\right).\]
Consequently, for every $p \in \N$, there exists $C(p) > 0$ such that the bound
\[\Es{\left(\Wfn(\upsilon_ns) - \min_{s \le r \le t} \Wfn(\upsilon_nr)\right)^p}
\le C(p) \cdot (\sigma_n + \varrho_n)^p \cdot |t-s|^{p/2},\]
holds uniformly for $0 \le s < t \le 1$ with $|t - s| \le 1/2$.
\end{prop}

\begin{proof}
Note that in~\eqref{eq:queue_exp_min_pont_bis}, since $\min_{i \le \alpha \upsilon_n} \Bfn(i) \ge - \alpha \upsilon_n$, we may assume that $\sigma_n z \le \alpha (\upsilon_n - \varrho_n)$. From the preceding discussion, under our assumption the latter equals $\alpha \epsilon_n \le \alpha \delta^{-1} \sigma_n^2$, whence
\[\Pr{\min_{i \le \alpha \upsilon_n} \Bfn(i) + \alpha \varrho_n \le - \sigma_n z}
\le \exp\left(- \frac{(1 - \alpha)^3 \sigma_n z^2}{2 \alpha \sigma_n + \alpha \delta^{-1} \sigma_n}\right)
= \exp\left(- \frac{(1 - \alpha)^3 z^2}{\alpha (2 + \delta^{-1})}\right),\]
for all $z > 0$. Therefore, if $\alpha \le 1/2$, we have
\[\Pr{\min_{i \le \alpha \upsilon_n} \Bfn(i) + \alpha \varrho_n \le - \sigma_n \alpha^{1/2} z}
\le \exp\left(- \frac{z^2}{8 (2 + \delta^{-1})}\right),\]
for all $z > 0$. By exchangeability, we have for any $0 \le s < t \le 1$ fixed with $|t - s| \le 1/2$,
\[\Pr{\Bfn(\upsilon_ns) - \min_{s \le r \le t} \Bfn(\upsilon_nr) - \varrho_n |t-s| > \sigma_n |t-s|^{1/2} z} \le \exp\left(- \frac{z^2}{8 (2 + \delta^{-1})}\right).\]
We then transfer this bound to $\Wfn$ using that the latter is the cyclic shift of $\Bfn$ at the random time $i_n$, so one recover $\Bfn$ by cyclicly shifting $\Wfn$ at time $\upsilon_n - i_n$. The reader may want to consult Figure~6 in~\cite{Marzouk:Scaling_limits_of_discrete_snakes_with_stable_branching}, where $i_n$ is denoted by $a_n$. Fix $0 \le s < t \le 1$, there are two cases: either $\upsilon_n - i_n$ falls between $\upsilon_n s$ and $\upsilon_n t$, or it does not. If it does not, then the path of $\Wfn$ between $\upsilon_n s$ and $\upsilon_n t$ is moved without change in $\Bfn$ so the claim follows from the preceding bound applied to the image of $s$ and $t$ after the cyclic shift operation, which are still at distance $|t-s| \upsilon_n$ from each other. If $\upsilon_n - i_n$ does fall between $\upsilon_n s$ and $\upsilon_n t$, then the part of $\Wfn$ between $\upsilon_n s$ and $\upsilon_n - i_n$ is moved to the part of $\Bfn$ between $\upsilon_n s + i_n$ and $\upsilon_n$, and the part of $\Wfn$ between $\upsilon_n - i_n$ and $\upsilon_n t$ is moved to the part of $\Bfn$ between $0$ and $\upsilon_n t + i_n - \upsilon_n$. 
Using that for every $a, b \ge 0$, we have $(a+b)^{1/2} \ge 2^{-1/2} (a^{1/2} + b^{1/2})$, we then obtain that in this case, since $|t-s| = |1 - (s + \frac{i_n}{\upsilon_n})| + |t + \frac{i_n}{\upsilon_n} - 1|$,
\begin{align*}
&\Pr{\Wfn(\upsilon_ns) - \min_{s \le r \le t} \Wfn(\upsilon_nr) - \varrho_n |t-s| > \sigma_n |t-s|^{1/2} z}
\\
&\le \Pr{\Bfn(\upsilon_n s + i_n) - \min_{\upsilon_n s + i_n \le k \le \upsilon_n} \Bfn(k) - \varrho_n \left|1 - \left(s + \frac{i_n}{\upsilon_n}\right)\right| > \sigma_n \left|1 - \left(s + \frac{i_n}{\upsilon_n}\right)\right|^{1/2} \frac{z}{2^{1/2}}}
\\
&\quad+ \Pr{- \min_{0 \le k \le \upsilon_n t + i_n - \upsilon_n} \Bfn(k) - \varrho_n \left|t + \frac{i_n}{\upsilon_n} - 1\right| > \sigma_n \left|t + \frac{i_n}{\upsilon_n} - 1\right|^{1/2} \frac{z}{2^{1/2}}}
\\
&\le 2 \exp\left(- \frac{z^2}{16 (2 + \delta^{-1})}\right).
\end{align*}

Finally, by integrating this tail bound applied to $z^{1/p}$, we obtain
\[\Es{\left(\frac{\Wfn(\upsilon_ns) - \min_{s \le r \le t} \Wfn(\upsilon_nr) - \varrho_n |t-s|}{\sigma_n |t-s|^{1/2}}\right)^p} 
\le \int_0^\infty 2 \exp\left(- \frac{z^{2/p}}{16 (2 + \delta^{-1})}\right) \d z
\eqqcolon c(p).\]
Since $a^p + b^p \le (a+b)^p \le 2^{p-1} (a^p + b^p)$ for every $a,b \ge 0$, we conclude that
\begin{align*}
&\Es{\left(\Wfn(\upsilon_ns) - \min_{s \le r \le t} \Wfn(\upsilon_nr)\right)^p}
\\
&\le 2^{p-1} \left(\Es{\left(\Wfn(\upsilon_ns) - \min_{s \le r \le t} \Wfn(\upsilon_nr) - \varrho_n |t-s|\right)^p} + \left(\varrho_n |t-s|\right)^p\right)
\\
&\le 2^{p-1} \left(\sigma_n^p |t-s|^{p/2} c(p) + \varrho_n^p |t-s|^{p/2}\right)
\\
&\le \left(2^{p-1} (c(p) + 1)\right) (\sigma_n + \varrho_n)^p |t-s|^{p/2},
\end{align*}
and the proof is complete.
\end{proof}

The next result states that vertices with a given out-degree are homogeneously spread in the forest. 
For a subset $A \subset \Z_{\ge -1} = \{-1, 0, 1, 2, \dots\}$ and a real $r \in [0, \upsilon_n]$, let $\Lambda_A(r; \Wfn)$ denote the number of jumps with value in $A$ amongst the first $\lfloor r\rfloor$ jumps of $\Wfn$, and consider its inverse $\zeta_A(p; \Wfn) = \inf\{r : \Lambda_A(r; \Wfn) = \lfloor p\rfloor\}$ for all real $p \ge 1$. Let us set $\degr_n(A + 1) = \sum_{k \in A} \degr_n(k+1)$.

\begin{lem}\label{lem:proportion_feuilles}
Fix any subset $A \subset \Z_{\ge -1}$. If $\degr_n(A+1) \to \infty$, then we have the convergence in probability
\[\left(\frac{\Lambda_A(\upsilon_n t; \Wfn)}{\degr_n(A+1)}, \frac{\zeta_A(\degr_n(A+1) t; \Wfn)}{\upsilon_n}\right)_{t \in [0,1]} \cvproba (t, t)_{t \in [0,1]}.\]
\end{lem}

We shall apply this result to $A = \{-1\}$ at the very end of this paper. Note that since the two functions are inverse of one another and non-decreasing, the two convergences are equivalent.

\begin{proof}
Let us first prove the similar convergence when the {\L}ukasiewicz path $\Wfn$ is replaced by the bridge $\Bfn$. Fix $\varepsilon > 0$, then
\[\Pr{\max_{i \le \degr_n(A+1)} \left|\frac{\zeta_A(i; \Bfn)}{\upsilon_n} - \frac{i}{\degr_n(A+1)}\right| > \varepsilon}
\le \sum_{i \le \degr_n(A+1)} \Pr{\left|\zeta_A(i; \Bfn) - i \frac{\upsilon_n}{\degr_n(A+1)}\right| > \varepsilon \upsilon_n}.\]
Let us set $\Lambda_A(k; \Bfn) = \degr_n(A+1)$ if $k \ge \upsilon_n$ and $\Lambda_A(k; \Bfn) = 0$ if $k \le 0$. Note that $\zeta_A(i; \Bfn) > k$ if and only if $\Lambda_A(k; \Bfn) > i$, so the right-hand side above is bounded by
\begin{align*}
&\sum_{i \le \degr_n(A+1)} \Pr{\Lambda_A\left(i \frac{\upsilon_n}{\degr_n(A+1)} + \varepsilon \upsilon_n; \Bfn\right) \ge i} + \Pr{\Lambda_A\left(i \frac{\upsilon_n}{\degr_n(A+1)} - \varepsilon \upsilon_n; \Bfn\right) \le i}
\\
&\le \degr_n(A+1) \cdot \Pr{\exists j \le \upsilon_n : \Lambda_A\left(j; \Bfn\right) \ge (j - \varepsilon \upsilon_n) \frac{\degr_n(A+1)}{\upsilon_n}}
\\
&\quad+ \degr_n(A+1) \cdot \Pr{\exists j \le \upsilon_n : \Lambda_A\left(j; \Bfn\right) \le (j + \varepsilon \upsilon_n) \frac{\degr_n(A+1)}{\upsilon_n}}
\\
&\le 2 \degr_n(A+1) \cdot \Pr{\max_{j \le \upsilon_n} \left|\Lambda_A\left(j; \Bfn\right) - j \frac{\degr_n(A+1)}{\upsilon_n}\right| \ge \varepsilon \degr_n(A+1)}.
\end{align*}
The quantity $\Lambda_a(j; \Bfn)$ is the sum of dependent Bernoulli random variables: we start with an urn containing $\upsilon_n$ balls in total, amongst which $\degr_n(A+1)$ are labelled $A$, we sample $j$ balls without replacement and count the number of balls labelled $A$ picked. The Chernoff bound applies in the same way it does with i.i.d. Bernoulli coming from sampling with replacement (this is arleady shown in Hoeffding’s seminal paper~\cite[Theorem~4]{Hoeffding:Probability_inequalities_for_sums_of_bounded_random_variables}) so for every $\varepsilon > 0$, it holds that
\[\Pr{\max_{j \le \upsilon_n} \left|\Lambda_A(j; \Bfn) - j \frac{\degr_n(A+1)}{\upsilon_n}\right| > \varepsilon \degr_n(A+1)}
\le 2 \exp\left(- \frac{\varepsilon^2 \degr_n(A+1)}{2 + 2 \varepsilon / 3}\right).\]
We conclude that
\[\Pr{\max_{i \le \degr_n(A+1)} \left|\frac{\zeta_A(i; \Bfn)}{\upsilon_n} - \frac{i}{\degr_n(A+1)}\right| > \varepsilon}
\le 4 \degr_n(A+1) \exp\left(- \frac{\varepsilon^2 \degr_n(A+1)}{2 + 2 \varepsilon / 3}\right),\]
which converges to $0$ since we assume that $\degr_n(A+1) \to \infty$. This shows that
\[\left(\frac{\zeta_A(\degr_n(A+1) t; \Bfn)}{\upsilon_n}\right)_{t \in [0, 1]}
\qquad\text{and}\qquad
\left(\frac{\Lambda_A(\upsilon_n t; \Bfn)}{\degr_n(A+1)}\right)_{t \in [0, 1]},\]
both converge in probability to the identity.

In order to transfer these bounds from the bridge $\Bfn$ to the excursion $\Wfn$, we use the same reasoning as in the proof of Proposition~\ref{prop:moments_marche_Luka}, by constructing $\Wfn$ by cyclicly shifting $\Bfn$ at a uniformly random time which is independent of $\Wfn$, we leave the details to the reader.
\end{proof}

\subsection{Tail bounds for the width of a single tree}
\label{sec:queues_exp}

The results of this subsection will not be used here but in the companion paper~\cite{Marzouk:Scaling_limits_of_planar_trees_and_maps_with_a_prescribed_degree_sequence_Brownian_limits}; the discussion also shows that the behaviour of the random forests $\fn$ is more complicated than the random maps $\mn$. Indeed, we claim in Theorem~\ref{thm:tension_cartes} that the latter are always tight, but as we discuss below, even in the case of a single tree, this is not the case of $\fn$ in full generality.

\begin{prop}\label{prop:queues_exp_Luka}
Suppose that $\Delta_n \ge 2$. Let $x_n$ be a uniformly random vertex of $\fn$, then 
\[\Pr{\LR(x_n) \ge z \sigma_n} \le 4 \exp\left(- \frac{z}{288}\right)\]
uniformly in $z \ge 1/2$ and $n \in \N$.
\end{prop}

\begin{proof}
Recall that $\LR(x_n) = \LL(x_n) + \RR(x_n)$ denotes the number of those vertices branching-off of the ancestral line of $x_n$ in the tree of $\fn$ which contains it. By symmetry, considering the `mirror forest' obtained from $\fn$ by flipping the order of the children of every vertex, $\LL(x_n)$ and $\RR(x_n)$ have the same law so it suffices to consider the tails of $\RR(x_n) = \Wfn(x_n) - \min_{0 \le y \le x_n} \Wfn(y)$. Recall also that $\Wfn$ can be obtained by cyclicly shifting the bridge $\Bfn$ at the random time $i_n$ which is uniformly distributed in $\{1, \dots, \upsilon_n\}$ and is independent of $\Wfn$. In this coupling, we have that
\[\Wfn(\upsilon_n - i_n) - \min_{0 \le j \le \upsilon_n - i_n} \Wfn(j)
= \Bfn(\upsilon_n) - \min_{0 \le j \le \upsilon_n} \Bfn(j)
= -\varrho_n - \min_{0 \le j \le \upsilon_n} \Bfn(j).\]
Note that $\upsilon_n - i_n$ is uniformly distributed in $\{0, \dots, \upsilon_n - 1\}$ and is independent of $\Wfn$, whereas $x_n$ is the vertex visited at a random time uniformly distributed in $\{1, \dots, \upsilon_n\}$ and independently of $\Wfn$. Since $\Wfn(0) - \min_{0 \le j \le 0} \Wfn(j) = \Wfn(\upsilon_n) - \min_{0 \le j \le \upsilon_n} \Wfn(j) = 0$, we have that $\RR(x_n)$ has the same law as $-\varrho_n - \min_{0 \le j \le \upsilon_n} \Bfn(j)$. 
Let us next observe the following bound
\[-\varrho_n - \min_{0 \le j \le \upsilon_n} \Bfn(j)
\le \left(- \frac{\varrho_n}{2} - \min_{0 \le j \le \upsilon_n/2} \Bfn(j)\right) + \left(- \frac{\varrho_n}{2} - \min_{\upsilon_n/2 \le j \le \upsilon_n} \Bfn(j) + \Bfn\left(\frac{\upsilon_n}{2}\right)\right),\]
and the two terms on the right have the same law. From our assumption, we have $\sigma_n^2 \ge \Delta (\Delta_n - 1) \ge 1$ so by~\eqref{eq:queue_exp_min_pont_bis}, for every $z \ge 1/4$ we have after two union bounds
\[\Pr{\LR(x_n) \ge z \sigma_n}
\le 4 \cdot \Pr{\min_{0 \le j \le \upsilon_n/2} \Bfn(j) + \frac{\varrho_n}{2} \le - \frac{z \sigma_n}{4}}
\le 4 \exp\left(- \frac{2^{-3} \sigma_n (z/4)^2}{\sigma_n + z/4}\right).\]
Since both $z \ge \frac{1}{2}$ and $\sigma_n \ge 1$, we have that $\sigma_n + \frac{z}{4} \le \frac{9}{4} \sigma_n z$, which yields the bound in our claim.
\end{proof}

Let us next focus on the case of a single tree $\varrho_n = 1$. Let us denote by $\wid(\tn)$ the maximum over all $i \ge 1$ of the number of vertices of $\tn$ at distance $i$ from the root, called the \emph{width} of $\tn$, and by $\h(\tn)$ the greatest $i \ge 1$ such that there exists at least one vertex of $\tn$ at distance $i$ from the root, called the \emph{height} of $\tn$. the preceding result gives a universal tail bound on the first quantity.

\begin{cor}
For every $n \in \N$, for every degree sequence $\degr_n$ such that $\Delta_n \ge 2$, we have for every $z \ge 1$:
\[\Pr{\wid(\tn) \ge z \sigma_n} \le 3 \e^{-z / 48}.\]
\end{cor}

This tail bound slightly improves, when the forests are long and thin, on the one obtained by Addario-Berry~\cite{Addario_Berry:Tail_bounds_for_the_height_and_width_of_a_random_tree_with_a_given_degree_sequence} for a single tree which considered the larger scaling factor given by the Euclidian norm of the degrees $(\sum_{k \ge 1} k^2 \degr_n(k))^{1/2} = (\sigma_n^2 + \epsilon_n)^{1/2}$. The proof in~\cite{Addario_Berry:Tail_bounds_for_the_height_and_width_of_a_random_tree_with_a_given_degree_sequence} (see also~\cite[Theorem~9]{Kortchemski:Sub_exponential_tail_bounds_for_conditioned_stable_Bienayme_Galton_Watson_trees}) is a direct application of Theorem~3 there; Equation~3 there does not depend on the degrees, and Equation~2 may be replaced by the bound
\[\Pr{\min_{0 \le j \le \upsilon_n/2} \Btn(j) + \frac{1}{2} \le - z \sigma_n}
\le \exp\left(- \frac{z}{16}\right),\]
for all $z \ge 1$, which follows from~\eqref{eq:queue_exp_min_pont_bis} as in the preceding proof, which yields our corollary.

Thanks to the simple observation that $\wid(\tn) \times \h(\tn) \ge \epsilon_n$, we obtain the lower bound
\[\Pr{\h(\tn) \le \epsilon_n / (z \sigma_n)} \le 3 \e^{-z / 48},\]
for all $z \ge 1$. Following~\cite{Addario_Berry:Tail_bounds_for_the_height_and_width_of_a_random_tree_with_a_given_degree_sequence} we may also obtain
\begin{equation}\label{eq:queue_exp_hauteur}
\Pr{\h(\tn) \ge z \sigma_n} \le c_1 \e^{-c_2 z},
\end{equation}
for some constants $c_1, c_2 > 0$, but the scalings $\sigma_n$ and $\epsilon_n / \sigma_n$ are of the same order only in a `finite-variance' regime, when $\sigma_n$ is of order $\epsilon_n^{1/2}$, otherwise the former is larger than the latter. We stress that an upper-bound at the scaling $\epsilon_n / \sigma_n$ in full generality is not possible. As a matter of fact, the height of the tree may be much larger than $\epsilon_n / \sigma_n$ as it is the case for subcritical size-conditioned Bienaymé--Galton--Watson trees, for which $\epsilon_n / \sigma_n$ is of constant order, but the height grows like $\log \epsilon_n$, see Kortchemski~\cite{Kortchemski-Limit_theorems_for_conditioned_non_generic_Galton_Watson_trees}. Nonetheless, the scaling $\epsilon_n / \sigma_n$ is correct for typical vertices, see Proposition~\ref{prop:queues_exp_hauteur_typique} below.

\section{A spinal decomposition}
\label{sec:epine}

We describe in this section the ancestral lines of i.d.d. random vertices of $\fn$. A related result on trees was first obtained by Broutin \& Marckert~\cite{Broutin-Marckert:Asymptotics_of_trees_with_a_prescribed_degree_sequence_and_applications} for a single random vertex and it was extended in~\cite{Marzouk:Scaling_limits_of_random_bipartite_planar_maps_with_a_prescribed_degree_sequence} to several vertices. This present one is different, we shall compare them after the statement.

Suppose that an urn contains initially $k \degr_n(k)$ balls labelled $k$ for every $k \ge 1$, so $\epsilon_n$ balls in total; let us pick balls repeatedly one after the other \emph{without} replacement; for every $1 \le i \le \epsilon_n$, we denote the label of the $i$-th ball by $\xi_{\degr_n}(i)$. Conditional on $(\xi_{\degr_n}(i))_{1 \le i \le \epsilon_n}$, let us sample independent random variables $(\chi_{\degr_n}(i))_{1 \le i \le \epsilon_n}$ such that each $\chi_{\degr_n}(i)$ is uniformly distributed in $\{1, \dots, \xi_{\degr_n}(i)\}$.
We shall use the fact that the $\xi_{\degr_n}(i)$'s are identically distributed, with
\begin{equation}\label{eq:moyenne_biais_par_la_taille}
\Es{\xi_{\degr_n}(i) - 1} = \sum_{k \ge 1} (k-1) \frac{k \degr_n(k)}{\epsilon_n} = \frac{\sigma_n^2}{\epsilon_n}
\qquad\text{and so}\qquad
\Var\left(\xi_{\degr_n}(i)-1\right) \le \Delta_n \frac{\sigma_n^2}{\epsilon_n}.
\end{equation}

Let us fix a plane forest $F$, viewed as a forest and not a tree. Recall that we denote by $\chi_{x}$ the relative position of a vertex $x$ amongst its siblings; for every $0 \le i \le |x|$, let us also denote by $a_i(x)$ the ancestor of $x$ at height $i$, so $a_{|x|}(x) = x$ and $a_0(x)$ is the root of the tree containing $x$. Define next for every vertex $x$ the \emph{content} of the branch $\mathopen{\llbracket}a_0(x), x\mathclose{\llbracket}$ as
\begin{equation}\label{eq:content}
\Cont(x) = \left(\left(k_{a_{i-1}(x)}, \chi_{a_i(x)}\right); 1 \le i \le |x|\right).
\end{equation}
In words, $\Cont(x)$ lists the number of children of each strict ancestor of $x$ and the position of the next ancestor amongst these children. More generally, let $x_1, \dots, x_q$ be $q$ vertices of $F$ and let us consider the forest $F$ \emph{reduced} to its root and these vertices: $F(x_1, \dots, x_q)$ contains only the vertices $x_1, \dots, x_q$ and their ancestors, and it naturally inherits a plane forest structure from $F$; let us further remove all the vertices with at least two children \emph{in this forest} to produce $\tilde{F}(x_1, \dots, x_q)$ which is a collection of single branches, and let $\Cont(x_1, \dots, x_q)$ be the sequence of the pairs $(k_{pr(y)}, \chi_y)$ where $y$ ranges over $\tilde{F}(x_1, \dots, x_q)$ in lexicographical order, and where the quantities $k_{pr(y)}$ and $\chi_y$ are those \emph{in the original forest $F$}.

\begin{lem}\label{lem:multi_epines_sans_remise}
Fix $n \in \N$ and $\degr_n$ a degree sequence such that $\Delta_n \ge 2$.
Let $q \ge 1$ and sample $x_{n,1}, \dots, x_{n,q}$ independently uniformly at random in $\fn$. Let $(k_i, j_i)_{1 \le i \le h}$ be positive integers such that $1 \le j_i \le k_i$ for each $i$. If $q \ge 2$, assume that $h, q \le \upsilon_n/4$. Then for every integers $0 \le b \le q-1$ and $1 \le c \le q$, the probability that $\Cont(x_{n,1}, \dots, x_{n,q}) = (k_i, j_i)_{1 \le i \le h}$, and that the reduced forest ${T}_{\degr_n}^{\varrho_n}(x_{n,1}, \dots, x_{n,q})$ possesses $c$ trees, $q$ leaves, and $b$ branch-points is bounded above by
\[q^2 2^{q - 1} \left(\frac{\sigma_n}{\upsilon_n}\right)^{q + b} \left(\frac{\varrho_n}{\sigma_n}\right)^{c-1} \frac{\varrho_n + (q-1) \Delta_n + \sum_{1 \le i \le h} (k_i - 1)}{\sigma_n}
\cdot \Pr{\bigcap_{i \le h} \left\{(\xi_{\degr_n}(i), \chi_{\degr_n}(i)) = (k_i, j_i)\right\}}.\]
\end{lem}

Note that the reduced forest ${T}_{\degr_n}^{\varrho_n}(x_{n,1}, \dots, x_{n,q})$ possesses $q$ leaves when no $x_{n,i}$ is an ancestor of another; we shall see below that the height of a random vertex is at most of order $\epsilon_n / \sigma_n$, so indeed with high probability no $x_{n,i}$ is an ancestor of another as soon as $\sigma_n \to \infty$. Also, if the forest has $b$ branch-points, then $q+b$ denotes the number of branches once we remove these branch-points; these branches typically have length of order $\epsilon_n/\sigma_n$, so the factor $(\sigma_n/\upsilon_n)^{q+b}$ is important. The other factors will be typically bounded in our applications.

In words, roughly speaking, along distinguished paths (removing the branch-points), up to a multiplicative factor, the individuals reproduce according to the \emph{size-biased law} $(k\degr_n(k))_{k \ge 1}$, and conditional on the offsprings, the paths continue via one of the children chosen uniformly at random. Note that these size-biased picks are not independent, since we are sampling without replacement. In the case of a single tree $\varrho_n = 1$, an analogous result when sampling \emph{with} replacement was obtained by Broutin \& Marckert~\cite{Broutin-Marckert:Asymptotics_of_trees_with_a_prescribed_degree_sequence_and_applications} for a single random vertex and it was extended in~\cite{Marzouk:Scaling_limits_of_random_bipartite_planar_maps_with_a_prescribed_degree_sequence} to several vertices. The significant difference is that when comparing to sampling with replacement, an extra factor of order $\e^{h^2/\epsilon_n}$ appears in the bound, and one cannot remove it. This was not an issue in~\cite{Broutin-Marckert:Asymptotics_of_trees_with_a_prescribed_degree_sequence_and_applications, Marzouk:Scaling_limits_of_random_bipartite_planar_maps_with_a_prescribed_degree_sequence} which focus on the `finite-variance regime', when $\sigma_n^2$ is of order $\epsilon_n$ since then the bound~\eqref{eq:queue_exp_hauteur} on the height ensures that it does not explode but this is not the case here. Moreover, we feel that sampling without replacement is more natural in this model.

We stress that the fact that the random variables $\xi_{\degr_n}(i)$ are sampled without replacement is not a technical issue. As a matter of fact, it is known, see e.g.~\cite[Proposition~20.6]{Aldous:Saint_Flour}, that these are more concentrated than their version sampled with replacement $\xi^\ast_{\degr_n}(i)$ in the sense that for any integer $h$ and any convex function $f$, it holds that $\E[f(\sum_{1 \le i \le h} \xi_{\degr_n}(i))] \le\E[f(\sum_{1 \le i \le h} \xi^\ast_{\degr_n}(i))]$. Therefore, the Markov inequality and any Chernoff-type bound based on controlling the Laplace transform applies to the $\xi_{\degr_n}(i)$'s in the same way it applies to the $\xi^\ast_{\degr_n}(i)$'s.

\subsection{Proof of the one-point decomposition}

Let us first consider the simpler case of a single random vertex. In this case, we have $b=0$ and $c=1$ so the upper bound reads simply
\begin{equation}\label{eq:epine_un_sommet}
\frac{\varrho_n + \sum_{1 \le i \le h} (k_i - 1)}{\upsilon_n}
\cdot \Pr{\bigcap_{i \le h} \left\{(\xi_{\degr_n}(i), \chi_{\degr_n}(i)) = (k_i, j_i)\right\}}.
\end{equation}

\begin{proof}[Proof of Lemma~\ref{lem:multi_epines_sans_remise} for $q = 1$]
The proof relies on the fact that there is a one-to-one correspondence between a plane forest with $\varrho$ trees and a distinguished vertex $x$ on the one hand, and on the other hand the triplet given by the knowledge of which of the $\varrho$ trees contains $x$, the vector $\Cont(x)$, and the plane forest obtained by removing all the strict ancestors of $x$; note that this forest contains $\varrho + \sum_{0 \le i < |x|} (k_{a_i(x)} - 1)$ trees.
Recall that for any sequence $\theta = (\theta_\ell)_{\ell \ge 0}$ of non-negative integers with finite sum $|\theta|$, the number of plane forests having exactly $\theta_\ell$ vertices with $\ell$ children for every $\ell \ge 0$ is given by
\[\# \ensembles{F}(\theta) = \frac{r}{|\theta|} \binom{|\theta|}{(\theta_\ell)_{\ell \ge 0}},\]
where $r= \sum_{\ell \ge 0} (1-\ell) \theta_\ell$ is the number of roots.
Fix $(k_i, j_i)_{1 \le i \le h}$ a sequence of positive integers such that $1 \le j_i \le k_i$ for each $i$. For every $\ell \ge 1$, let $m_\ell = \#\{1 \le i \le h : k_i = \ell\}$ and assume that $m_\ell \le \degr_n(\ell)$; let also $m_0 = 0$ and $\m = (m_\ell)_{\ell \ge 0}$. 
By the preceding bijection, if $x_n$ is a vertex chosen uniformly at random, then we have
\begin{align*}
\Pr{\Cont(x_n) = (k_i, j_i)_{1 \le i \le h}}
&= \frac{\varrho_n \#\ensembles{F}(\degr_n - \m)}{\upsilon_n \#\ensembles{F}(\degr_n)}
\\
&= \frac{\varrho_n \frac{\varrho_n + \sum_{1 \le i \le h} (k_i - 1)}{\upsilon_n-h} \binom{\upsilon_n-h}{(\degr_n(\ell)-m_\ell)_{\ell\ge 0}}}{\upsilon_n \frac{\varrho_n}{\upsilon_n} \binom{\upsilon_n}{(\degr_n(\ell))_{\ell\ge 0}}}
\\
&= \frac{\varrho_n + \sum_{1 \le i \le h} (k_i - 1)}{\upsilon_n-h} \frac{(\upsilon_n-h)!}{\upsilon_n!} \prod_{\ell \ge 1} \frac{\degr_n(\ell)!}{(\degr_n(\ell)-m_\ell)!}.
\end{align*}
Let us set
\begin{align*}
P((k_i, j_i)_{i \le h}) &= \Pr{\bigcap_{i \le h} \left\{(\xi_{\degr_n}(i), \chi_{\degr_n}(i)) = (k_i, j_i)\right\}} 
\\
&= \frac{\prod_{\ell \ge 1} \ell^{- m_\ell} (\ell \degr_n(\ell)) (\ell \degr_n(\ell) - 1) \cdots (\ell \degr_n(\ell) - m_\ell + 1)}{\epsilon_n (\epsilon_n - 1) \cdots (\epsilon_n - h + 1)}
\\
&= \frac{(\epsilon_n - h)!}{\epsilon_n!} \prod_{\ell \ge 1} \ell^{- m_\ell} \frac{(\ell \degr_n(\ell))!}{(\ell \degr_n(\ell) - m_\ell)!}.
\end{align*}
Then, we have
\begin{multline*}
\Pr{\Cont(u_n) = (k_i, j_i)_{1 \le i \le h}}
\\
= \frac{\varrho_n + \sum_{1 \le i \le h} (k_i - 1)}{\upsilon_n-h} \frac{(\upsilon_n-h)!}{\upsilon_n!}
\frac{\epsilon_n!}{(\epsilon_n - h)!} 
\left(\prod_{\ell \ge 1} \ell^{m_\ell} \frac{\degr_n(\ell)! (\ell \degr_n(\ell) - m_\ell)!}{(\degr_n(\ell)-m_\ell)! (\ell \degr_n(\ell))!}\right)
P((k_i, j_i)_{i \le h}).
\end{multline*}
Observe that $\upsilon_n \ge \epsilon_n + 1$, so
\[\frac{(\upsilon_n-h)!}{\upsilon_n!} \frac{\epsilon_n!}{(\epsilon_n - h)!}
= \frac{\upsilon_n-h}{\upsilon_n} \prod_{i=0}^{h-1} \frac{\epsilon_n - i}{\upsilon_n - 1 - i}
\le \frac{\upsilon_n-h}{\upsilon_n}.\]
Moreover, we claim that $\ell^{m_\ell} \frac{\degr_n(\ell)! (\ell \degr_n(\ell) - m_\ell)!}{(\degr_n(\ell)-m_\ell)! (\ell \degr_n(\ell))!} \le 1$ for every $\ell \ge 1$. Indeed, it equals $1$ when $\ell=1$; we suppose next that $\ell \ge 2$, so in particular $m_\ell \le \degr_n(\ell) \le \ell \degr_n(\ell) / 2$. It is simple to check that for every $x \in [0, 1/2]$, we have $(1-x)^{-1} \le 2^{2x}$, and for every $x \in [0, 1]$, we have $1-x \le 2^{-x}$. It follows that for every $\ell \ge 2$ such that $\degr_n(\ell) \ne 0$,
\begin{align*}
\ell^{m_\ell} \frac{\degr_n(\ell)! (\ell \degr_n(\ell)-m_\ell)!}{(\degr_n(\ell)-m_\ell)! (\ell \degr_n(\ell))!}
&= \frac{\degr_n(\ell)!}{(\degr_n(\ell)-m_\ell)! \degr_n(\ell)^{m_\ell}} \frac{(\ell \degr_n(\ell)-m_\ell)! (\ell \degr_n(\ell))^{m_\ell}}{(\ell \degr_n(\ell))!}
\\
&= \prod_{i = 0}^{m_\ell-1} \frac{\degr_n(\ell) - i}{\degr_n(\ell)} \frac{\ell \degr_n(\ell)}{\ell \degr_n(\ell) - i}
\\
&\le \prod_{i = 0}^{m_\ell-1} 
\exp\left(\ln 2 \left(-\frac{i}{\degr_n(\ell)} + \frac{2i}{\ell \degr_n(\ell)}\right)\right),
\end{align*}
which is indeed bounded by $1$ for $\ell \ge 2$. 
This concludes the proof in the case $q=1$.
\end{proof}

\subsection{Proof of the multi-point decomposition}

We next turn to the proof in the general case; the difference is that one has to deal with the branch-points of the reduced forest, which, we recall, are not considered in the vector of content.

\begin{proof}[Proof of Lemma~\ref{lem:multi_epines_sans_remise} for $q \ge 2$]
Fix $q \ge 2$ and sample $x_{n,1}, \dots, x_{n,q}$ i.i.d. uniformly random vertices in $\fn$. As in the case $q=1$, there is a one-to-one correspondence between the plane forest $\fn$ and $x_{n,1}, \dots, x_{n,q}$ on the one hand, and on the other hand the plane forest obtained by removing $\fn(x_{n,1}, \dots, x_{n,q})$ from the whole forest $\fn$, the vector $\Cont(x_{n,1}, \dots, x_{n,q})$, and the following data: first the number $c$ of trees containing at least one of the $x_{n,i}$'s and the knowledge of which ones, second for each of the $b$ branch-points: their total number $r$ of children in $\fn$, their number $d$ of children which belong to $\fn(x_{n,1}, \dots, x_{n,q})$, and the relative positions $z_i$'s of these children.

Therefore, fix $(k_i, j_i)_{1 \le i \le h}$, also $c \in \{1, \dots, q\}$, $b \in \{0, \dots, q-1\}$, and for every $p \in \{1, \dots, b\}$, fix $r(p) \ge d(p) \ge 2$ and integers $1 \le z_{p,1} < \dots < z_{p, d(p)} \le r(p)$, and let us consider the following event: first $\Cont(x_{n,1}, \dots, x_{n,q}) = (k_i, j_i)_{1 \le i \le h}$, second, in the forest spanned by $x_{n,1}, \dots, x_{n,q}$, we have $c$ trees, $b$ branch-points, for every $p \le b$, the $p$-th branch-point in lexicographical order has $r(p)$ children in total in the original forest $\fn$, amongst which $d(p)$ belong to the reduced forest, and the relative positions of the latter are given by the $z_{p,i}$'s.
If we set $m_0 = 0$ and for every $\ell \ge 1$, we let $m_\ell = \#\{1 \le i \le h : k_i = \ell\}$ and $\overline{m}_\ell = m_\ell + \sum_{1 \le p \le b} \ind{r(p) = \ell}$, then on this event, the complement of the reduced forest is a forest with degree sequence $(\degr_n(\ell)-\overline{m}_\ell)_{\ell\ge 0}$; note that it contains $R = \varrho_n - c + q + \sum_{1 \le p \le b} (r(p) - d(p)) + \sum_{1 \le i \le h} (k_i - 1)$ trees and $\upsilon_n - h - b$ vertices. Therefore the probability of our event is given by
\[\frac{\binom{\varrho_n}{c} \#\ensembles{F}(\degr_n - \overline{\m})}{\upsilon_n^q \#\ensembles{F}(\degr_n)}
= \frac{\binom{\varrho_n}{c} \frac{R}{\upsilon_n-h-b} \binom{\upsilon_n-h-b}{(\degr_n(\ell)-\overline{m}_\ell)_{\ell\ge 0}}}{\upsilon_n^q \frac{\varrho_n}{\upsilon_n} \binom{\upsilon_n}{(\degr_n(\ell))_{\ell\ge 0}}}
= \frac{\binom{\varrho_n}{c} R}{\varrho_n \upsilon_n^{q-1} (\upsilon_n-h-b)}
\frac{(\upsilon_n-h-b)!}{\upsilon_n!}
\prod_{\ell \ge 1} \frac{(\degr_n(\ell))!}{(\degr_n(\ell)-\overline{m}_\ell)!}.\]
Now as previously, we have
\[\frac{(\upsilon_n-h-b)!}{\upsilon_n!}
= \frac{(\upsilon_n-h-b)!}{(\upsilon_n - h)!} \frac{(\upsilon_n-h)!}{\upsilon_n!}
\le \frac{(\upsilon_n-h-b)!}{(\upsilon_n - h)!} \frac{\upsilon_n-h}{\upsilon_n} \frac{(\epsilon_n - h)!}{\epsilon_n!}
= \frac{(\upsilon_n-h-b)!}{\upsilon_n (\upsilon_n - h - 1)!} \frac{(\epsilon_n - h)!}{\epsilon_n!}.\]
Next,
\begin{align*}
\prod_{\ell \ge 1} \frac{(\degr_n(\ell))!}{(\degr_n(\ell)-\overline{m}_\ell)!}
&= \prod_{\ell \ge 1} \frac{(\degr_n(\ell))!}{(\degr_n(\ell)-m_\ell)!}
\prod_{\ell \ge 1} \frac{(\degr_n(\ell)-m_\ell)!}{(\degr_n(\ell)-m_\ell-\sum_{1 \le p \le b} \ind{r(p) = \ell})!}
\\
&\le \prod_{\ell \ge 1} \frac{(\degr_n(\ell))!}{(\degr_n(\ell)-m_\ell)!}
\prod_{1 \le p \le b} \degr_n(r(p)).
\end{align*}
Finally, we have seen that
\[P((k_i, j_i)_{i \le h}) \coloneqq \Pr{\bigcap_{i \le h} \left\{(\xi_{\degr_n}(i), \chi_{\degr_n}(i)) = (k_i, j_i)\right\}} 
= \frac{(\epsilon_n - h)!}{\epsilon_n!} \prod_{\ell \ge 1} \ell^{- m_\ell} \frac{(\ell \degr_n(\ell))!}{(\ell \degr_n(\ell) - m_\ell)!},\]
so we obtain
\begin{align*}
\frac{\binom{\varrho_n}{c} \#\ensembles{F}(\degr_n - \m)}{\upsilon_n \#\ensembles{F}(\degr_n)}
&\le \frac{\binom{\varrho_n}{c} R}{\varrho_n \upsilon_n^{q-1} (\upsilon_n-h-b)}
\frac{(\upsilon_n-h-b)!}{\upsilon_n (\upsilon_n - h - 1)!} \frac{(\epsilon_n - h)!}{\epsilon_n!}
\prod_{\ell \ge 1} \frac{(\degr_n(\ell))!}{(\degr_n(\ell)-m_\ell)!}
\prod_{1 \le p \le b} \degr_n(r(p))
\\
&\le \frac{\binom{\varrho_n}{c} R (\upsilon_n-h-b-1)!}{\varrho_n \upsilon_n^q (\upsilon_n - h - 1)!}
\prod_{\ell \ge 1} \ell^{m_\ell} \frac{(\ell \degr_n(\ell) - m_\ell)! (\degr_n(\ell))!}{(\ell \degr_n(\ell))! (\degr_n(\ell)-m_\ell)!}
\prod_{1 \le p \le b} \degr_n(r(p))
\cdot P((k_i, j_i)_{i \le h})
\\
&\le \frac{\binom{\varrho_n}{c} R (\upsilon_n-h-b-1)!}{\varrho_n \upsilon_n^q (\upsilon_n - h - 1)!}
\prod_{1 \le p \le b} \degr_n(r(p))
\cdot P((k_i, j_i)_{i \le h}),
\end{align*}
where the last bound was shown in the poof of the case $q=1$. Since $c \ge 1$, $b \le q-1$, and $2 \le d(p) \le r(p) \le \Delta_n$, we have
\begin{align*}
R &= \varrho_n - c + q + \sum_{1 \le p \le b} (r(p) - d(p)) + \sum_{1 \le i \le h} (k_i - 1)
\\
&\le \varrho_n - 1 + q + b (\Delta_n-1) + \sum_{1 \le i \le h} (k_i - 1)
\\
&\le \varrho_n + (q-1) \Delta_n + \sum_{1 \le i \le h} (k_i - 1).
\end{align*}
Moreover, from our assumption that $h, q \le \upsilon_n/4$,
\[\frac{(\upsilon_n-h-b-1)!}{(\upsilon_n - h - 1)!}
\le \left(\frac{1}{\upsilon_n - h - b}\right)^b
\le \left(\frac{2}{\upsilon_n}\right)^b,\]
and finally, since $\binom{\varrho_n}{c} \le \varrho_n^c$, we obtain
\[\frac{\binom{\varrho_n}{c} \#\ensembles{F}(\degr_n - \m)}{\upsilon_n \#\ensembles{F}(\degr_n)}
\le 2^b \varrho_n^c \frac{\varrho_n + (q-1) \Delta_n + \sum_{1 \le i \le h} (k_i - 1)}{\varrho_n \upsilon_n^{q+b}}
\prod_{1 \le p \le b} \degr_n(r(p))
\cdot P((k_i, j_i)_{i \le h}).\]
On the left is the probability that $\Cont(x_{n,1}, \dots, x_{n,q}) = (k_i, j_i)_{1 \le i \le h}$ and that the reduced forest has $c$ trees, $b$ branch-points and that for every $p \le b$, the $p$-th branch-point has $r(p)$ children amongst which $d(p) \ge 2$ belong to the reduced forest, with relative positions given by $1 \le z_{p,1} < \dots < z_{p, d(p)} \le r(p)$. 
We now want to sum over all quantities besides $(k_i, j_i)_{1 \le i \le h}$; we consider them one after the other, from the last to the first. Indeed, the sum of the right-hand side above over all the possible $z$'s is bounded by
\begin{align*}
&\prod_{1 \le p \le b} \binom{r(p)}{d(p)} \degr_n(r(p))
\cdot 2^b \varrho_n^c \frac{\varrho_n + (q-1) \Delta_n + \sum_{1 \le i \le h} (k_i - 1)}{\varrho_n \upsilon_n^{q+b}}
\cdot P((k_i, j_i)_{i \le h})
\\
&\le \prod_{1 \le p \le b} \Delta_n^{d(p)-2} r(p)(r(p) - 1) \degr_n(r(p))
\cdot 2^b \varrho_n^c \frac{\varrho_n + (q-1) \Delta_n + \sum_{1 \le i \le h} (k_i - 1)}{\varrho_n \upsilon_n^{q+b}}
\cdot P((k_i, j_i)_{i \le h})
\\
&\le \Delta_n^{q - c - b} \prod_{1 \le p \le b} r(p)(r(p) - 1) \degr_n(r(p))
\cdot 2^b \varrho_n^c \frac{\varrho_n + (q-1) \Delta_n + \sum_{1 \le i \le h} (k_i - 1)}{\varrho_n \upsilon_n^{q+b}}
\cdot P((k_i, j_i)_{i \le h}),
\end{align*}
where for the last line, we note that, since the $d(p)$'s are the number of children of the branch-points in the reduced forest, which contains $c$ trees and $q$ leaves, then $q = c + \sum_{1 \le p \le b} (d(p)-1) = c + b + \sum_{1 \le p \le b} (d(p)-2)$. Note the very crude bound: there are less than $b q \le q^2$ such vectors $(d(1), \dots, d(b))$. We next want to sum the last bound (times $q^2$) over all the vectors $(r(1), \dots, r(b))$; we have
\[\sum_{r(1), \dots, r(b) \le \Delta_n} \prod_{1 \le p \le b} r(p)(r(p) - 1) \degr_n(r(p))
= \bigg(\sum_{r \le \Delta_n} r (r-1) \degr_n(r)\bigg)^b
= \sigma_n^{2b}.\]
If $\Delta_n \ge 2$, then $\Delta_n^2 \le 2 \Delta_n (\Delta_n - 1) \le 2\sigma_n^2$, so finally the probability that $\Cont(x_{n,1}, \dots, x_{n,q}) = (k_i, j_i)_{1 \le i \le h}$ and the reduced forest has $c \ge 1$ trees and $b$ branch-points is bounded by
\begin{align*}
& q^2 \Delta_n^{q - c - b} \sigma_n^{2b}
\cdot 2^b \varrho_n^c \frac{\varrho_n + (q-1) \Delta_n + \sum_{1 \le i \le h} (k_i - 1)}{\varrho_n \upsilon_n^{q+b}}
\cdot P((k_i, j_i)_{i \le h})
\\
&\le q^2 2^{q - c} \sigma_n^{q - c + b} \varrho_n^c \frac{\varrho_n + (q-1) \Delta_n + \sum_{1 \le i \le h} (k_i - 1)}{\varrho_n \upsilon_n^{q+b}}
\cdot P((k_i, j_i)_{i \le h})
\\
&\le q^2 2^{q - 1} \left(\frac{\sigma_n}{\upsilon_n}\right)^{q + b} \left(\frac{\varrho_n}{\sigma_n}\right)^{c-1} \frac{\varrho_n + (q-1) \Delta_n + \sum_{1 \le i \le h} (k_i - 1)}{\sigma_n}
\cdot P((k_i, j_i)_{i \le h}),
\end{align*}
and the proof is complete.
\end{proof}

\subsection{Exponential tails for the height of a typical vertex}
\label{sec:queue_exp_hauteur}

We next derive a tail bound for the height of a typical vertex $x_n$ of $\fn$. This result shall be used in the companion paper~\cite{Marzouk:Scaling_limits_of_planar_trees_and_maps_with_a_prescribed_degree_sequence_Brownian_limits}.

\begin{prop}\label{prop:queues_exp_hauteur_typique}
There exists two universal constants $c_1, c_2 > 0$ such that the following holds: Suppose that $\Delta_n \ge 2$ and let $x_n$ be a uniformly random vertex of $\fn$, then its height $|x_n|$ satisfies
\[\Pr{|x_n| \ge z \epsilon_n / \sigma_n} \le c_1 \exp\left(- c_2 z\right)\]
uniformly in $z \ge 1$ and $n \in \N$.
\end{prop}

As we already discussed in Section~\ref{sec:queues_exp}, this bound does not hold for the maximal height of the forest, since the latter may be much larger than $\epsilon_n / \sigma_n$.

Our argument is the following: if $x_n$ is at height at least $z \epsilon_n / \sigma_n$ and if its ancestors reproduce according to the size-biased law as the $\xi_{n,\degr_n}$'s in the preceding section, then, according to~\eqref{eq:moyenne_biais_par_la_taille}, the number of vertices $\LR(x_n)$ branching-off of its ancestral line is in average $(z \epsilon_n / \sigma_n) \times (\sigma_n^2 / \epsilon_n) = z \sigma_n$, and we known from Proposition~\ref{prop:queues_exp_Luka} that this has a sub-exponential cost.

\begin{proof}
According to Proposition~\ref{prop:queues_exp_Luka}, it is sufficient to bound
\[\Pr{|x_n| \ge z \frac{\epsilon_n}{\sigma_n} \text{ and } \LR(x_n) \le \frac{z}{2} \sigma_n}
= \sum_{h \ge z \epsilon_n / \sigma_n} \Pr{|x_n| = h \text{ and } \LR(x_n) \le \frac{z}{2} \sigma_n}.\]
Recall the definition of $\Cont(x_n) = ((k_{a_{i-1}(x_n)}, \chi_{a_i(x_n)}))_{1 \le i \le |x_n|}$ from~\eqref{eq:content} and note that we have $\LR(x_n) = \sum_{1 \le i \le |x_n|} (k_{a_{i-1}(x_n)} - 1)$. 
Let us fix $h \ge z \epsilon_n / \sigma_n$. We deduce from Lemma~\ref{lem:multi_epines_sans_remise} with a single random vertex that
\begin{align*}
& \Pr{|x_n| = h \text{ and } \LR(x_n) \le \frac{z}{2} \sigma_n}
\\
&= \sum_{(k_i, j_i)_{1 \le i \le h}} \Pr{\Cont(x_n) = (k_i, j_i)_{1 \le i \le h}} \ind{\sum_{1 \le i \le h} (k_i - 1) \le \frac{z}{2} \sigma_n}
\\
&\le \sum_{(k_i, j_i)_{1 \le i \le h}} 
\frac{1 + \sum_{1 \le i \le h} (k_i - 1)}{\upsilon_n} \cdot 
\Pr{\bigcap_{1 \le i \le h} \left\{(\xi_{\degr_n}(i), \chi_{\degr_n}(i)) = (k_i, j_i)\right\}} \ind{\sum_{1 \le i \le h} (k_i - 1) \le \frac{z}{2} \sigma_n}
\\
&\le \frac{\frac{z}{2} \sigma_n + 1}{\epsilon_n + 1} \Pr{\sum_{1 \le i \le h} \left(\xi_{\degr_n}(i) - 1\right) \le \frac{z}{2} \sigma_n}.
\end{align*}

Let us write $X_n(i) = \xi_{\degr_n}(i) - 1$ in order to simplify the notation. Recall from~\eqref{eq:moyenne_biais_par_la_taille} that these random variables have mean $\sigma_n^2 / \epsilon_n$ so for $h \ge z \epsilon_n / \sigma_n$, we have 
\[\Pr{|x_n| = h \text{ and } \LR(x_n) \le \frac{z}{2} \sigma_n}
\le \frac{z \sigma_n + 2}{2 \epsilon_n} \cdot 
\Pr{\sum_{1 \le i \le h} \left(X_n(i) - \E[X_n(i)]\right) \le - \frac{h \sigma_n^2}{2 \epsilon_n}}.\]
The $X_n(i)$'s come from sampling balls without replacement; as we already mentioned, by~\cite[Proposition~20.6]{Aldous:Saint_Flour}, their sum satisfies any concentration inequality based on controlling the Laplace transform the similar sum when sampling with replacement does. In particular, we may apply~\cite[Theorem~2.7]{McDiarmid:Concentration}, and get
\[\Pr{\sum_{1 \le i \le h} (X_n(i) - \E[X_n(i)]) \le - \frac{h \sigma_n^2}{2 \epsilon_n}}
\le \exp\left(- \frac{(\frac{h \sigma_n^2}{2 \epsilon_n})^2}{2 h \frac{\Delta_n \sigma_n^2}{\epsilon_n} + \frac{2}{3} \frac{h \sigma_n^2}{2 \epsilon_n} \frac{\sigma_n^2}{\epsilon_n}}\right)
= \exp\left(- \frac{h \sigma_n^2}{8 \epsilon_n \Delta_n + \frac{4}{3} \sigma_n^2}\right).\]
Observe that $\sigma_n^2 \le \epsilon_n \Delta_n$ so $8 \epsilon_n \Delta_n + \frac{4}{3} \sigma_n^2 \le 10 \epsilon_n \Delta_n$, whence
\begin{align*}
\Pr{|x_n| \ge z \frac{\epsilon_n}{\sigma_n} \text{ and } \LR(x_n) \le \frac{z}{2} \sigma_n}
&\le \frac{z \sigma_n + 2}{2 \epsilon_n} \sum_{h \ge z \epsilon_n / \sigma_n} \exp\left(- \frac{h \sigma_n^2}{10 \epsilon_n \Delta_n}\right)
\\
&\le \frac{z \sigma_n + 2}{2 \epsilon_n} 
\exp\left(- \frac{z \sigma_n}{10 \Delta_n}\right)
\left(1 - \exp\left(- \frac{\sigma_n^2}{10 \epsilon_n \Delta_n}\right)\right)^{-1}.
\end{align*}
We next appeal to the following two bounds: first $(1-\e^{-t}) \ge t(1 - t/2) \ge 19t/20$ for every $0 \le t \le 1/10$, second $t \e^{-t} \le \e^{-t/2}$ for all $t \ge 0$. We thus have
\begin{align*}
\Pr{|x_n| \ge z \text{ and } \LR(x_n) \le \frac{z}{2} \sigma_n}
&\le \frac{z \sigma_n + 2}{2 \epsilon_n}
\frac{10 \Delta_n}{z \sigma_n} \exp\left(- \frac{z \sigma_n}{20 \Delta_n}\right)
\frac{200 \epsilon_n \Delta_n}{19 \sigma_n^2}
\\
&\le \frac{1000 \Delta_n^2}{19 \sigma_n^2} \left(1 + \frac{2}{z \sigma_n}\right)
\exp\left(- \frac{z \sigma_n}{20 \Delta_n}\right).
\end{align*}
Recall that we assume that $\Delta_n \ge 2$, which implies $\Delta_n^2 \le 2 \sigma_n^2$ and also $\sigma_n \ge 1$. We thus obtain for every $z \ge 1$,
\[\Pr{|x_n| \ge z \text{ and } \LR(x_n) \le \frac{z}{2} \sigma_n}
\le \frac{2000 \times 3}{19} \exp\left(- \frac{z}{20 \sqrt{2}}\right).\]
Jointly with the exponential bound from Proposition~\ref{prop:queues_exp_Luka}, this completes the proof.
\end{proof}

\section{Maps and labelled forests}
\label{sec:arbres_etiquetes}

In this last section, we discuss random planar maps with prescribed face-degrees. The link with the rest of this work is via a bijection with \emph{labelled} forests which we first recall.

\subsection{Labelled trees and pointed maps}
\label{sec:bijection_arbre_carte}

For every $k \ge 1$, let us consider the following set of \emph{discrete bridges}
\[\mathscr{B}_k^{\ge -1} = \left\{(b_1, \dots, b_k): b_1, b_2-b_1, \dots, b_k-b_{k-1} \in \{-1, 0, 1, 2, \dots\} \text{ and } b_k=0\right\}.\]
Then a \emph{labelling} of a plane tree $T$ is a function $\ell$ from the vertices of $T$ to $\Z$ which satisfies that
\begin{enumerate}
\item the root of $T$ has label $\ell(\varnothing) = 0$,
\item for every vertex $x$ with $k_x \ge 1$ children, the sequence of increments $(\ell(x1)-\ell(x), \dots, \ell(xk_x)-\ell(x))$ belongs to $\mathscr{B}_{k_x}^{\ge -1}$.
\end{enumerate}
We encode then the labels into the \emph{label process} $L(k) = \ell(x_k)$, where $x_k$ is the $k$-th vertex of $T$ is lexicographical order; the labelled tree is encoded by the pair $(H, L)$. We extend the definition to forests by considering the associated tree, see Figure~\ref{fig:arbre_etiquete}.

\begin{figure}[!ht]\centering
\def\longueur{2}
\def\R{.5}
\def\r{1}
\begin{tikzpicture}[scale=.45]
\coordinate (1) at (0,0*\longueur);
	\coordinate (2) at (-4.25,1*\longueur);
	\coordinate (3) at (-1.5,1*\longueur);
	\coordinate (4) at (1.5,1*\longueur);
		\coordinate (5) at (.75,2*\longueur);
			\coordinate (6) at (.75,3*\longueur);
				\coordinate (7) at (-.75,4*\longueur);
				\coordinate (8) at (.25,4*\longueur);
				\coordinate (9) at (1.25,4*\longueur);
				\coordinate (10) at (2.25,4*\longueur);
		\coordinate (11) at (2.25,2*\longueur);
	\coordinate (12) at (4.25,1*\longueur);
		\coordinate (13) at (3.5,2*\longueur);
			\coordinate (14) at (2.5,3*\longueur);
			\coordinate (15) at (3.5,3*\longueur);
			\coordinate (16) at (4.5,3*\longueur);
		\coordinate (17) at (5,2*\longueur);

\draw[dashed]	(1) -- (2)	(1) -- (3)	(1) -- (4)	(1) -- (12);
\draw
	(4) -- (5)	(4) -- (11)
	(5) -- (6)
	(6) -- (7)	(6) -- (8)	(6) -- (9)	(6) -- (10)
	(12) -- (13)	(12) -- (17)
	(13) -- (14)	(13) -- (15)	(13) -- (16)
;

\begin{tiny}
\node[circle, minimum size=\R cm, fill=white, draw] at (2) {$-1$};
\node[circle, minimum size=\R cm, fill=white, draw] at (3) {$-2$};
\node[circle, minimum size=\R cm, fill=white, draw] at (7) {$-1$};
\node[circle, minimum size=\R cm, fill=white, draw] at (8) {$-2$};
\node[circle, minimum size=\R cm, fill=white, draw] at (9) {$-1$};
\node[circle, minimum size=\R cm, fill=white, draw] at (10) {$0$};
\node[circle, minimum size=\R cm, fill=white, draw] at (11) {$1$};
\node[circle, minimum size=\R cm, fill=white, draw] at (14) {$-2$};
\node[circle, minimum size=\R cm, fill=white, draw] at (15) {$0$};
\node[circle, minimum size=\R cm, fill=white, draw] at (16) {$-1$};
\node[circle, minimum size=\R cm, fill=white, draw] at (17) {$0$};

\node[circle, minimum size=\R cm, fill=white, draw] at (1) {$0$};
\node[circle, minimum size=\R cm, fill=white, draw] at (4) {$1$};
\node[circle, minimum size=\R cm, fill=white, draw] at (5) {$0$};
\node[circle, minimum size=\R cm, fill=white, draw] at (6) {$0$};
\node[circle, minimum size=\R cm, fill=white, draw] at (12) {$0$};
\node[circle, minimum size=\R cm, fill=white, draw] at (13) {$-1$};
\end{tiny}
\begin{footnotesize}
\begin{scope}[shift={(7,4)}]
\draw[thin, ->]	(0,0) -- (16.5,0);
\draw[thin, ->]	(0,0) -- (0,4.5);
\foreach \x in {1, 2, ..., 4}
	\draw[dotted]	(0,\x) -- (16.5,\x);
\foreach \x in {0, 1, ..., 4}
	\draw (.1,\x)--(-.1,\x)	(0,\x) node[left] {\x};
\foreach \x in {2, 4, ..., 16}
	\draw (\x,.1)--(\x,-.1)	(\x,0) node[below] {\x};

\draw[fill=black]
	(0, 0) circle (2pt) --
	++ (1, 1) circle (2pt) --
	++ (1, 0) circle (2pt) --
	++ (1, 0) circle (2pt) --
	++ (1, 1) circle (2pt) --
	++ (1, 1) circle (2pt) --
	++ (1, 1) circle (2pt) --
	++ (1, 0) circle (2pt) --
	++ (1, 0) circle (2pt) --
	++ (1, 0) circle (2pt) --
	++ (1, -2) circle (2pt) --
	++ (1, -1) circle (2pt) --
	++ (1, 1) circle (2pt) --
	++ (1, 1) circle (2pt) --
	++ (1, 0) circle (2pt) --
	++ (1, 0) circle (2pt) --
	++ (1, -1) circle (2pt)
;
\end{scope}
\begin{scope}[shift={(7,2)}]
\draw[thin, ->]	(0,0) -- (16.5,0);
\draw[thin, ->]	(0,-2) -- (0,1.5);
\foreach \x in {-2, -1, 1}
	\draw[dotted]	(0,\x) -- (16.5,\x);
\foreach \x in {-2, -1, 0, 1}
	\draw (.1,\x)--(-.1,\x)	(0,\x) node[left] {\x};
\foreach \x in {2, 4, ..., 16}
	\draw (\x,.1)--(\x,-.1)	;

\coordinate (0) at (0*\r, 0);
\coordinate (1) at (1*\r, -1);
\coordinate (2) at (2*\r, -2);
\coordinate (3) at (3*\r, 1);
\coordinate (4) at (4*\r, 0);
\coordinate (5) at (5*\r, 0);
\coordinate (6) at (6*\r, -1);
\coordinate (7) at (7*\r, -2);
\coordinate (8) at (8*\r, -1);
\coordinate (9) at (9*\r, 0);
\coordinate (10) at (10*\r, 1);
\coordinate (11) at (11*\r, 0);
\coordinate (12) at (12*\r, -1);
\coordinate (13) at (13*\r, -2);
\coordinate (14) at (14*\r, 0);
\coordinate (15) at (15*\r, -1);
\coordinate (16) at (16*\r, 0);
\newcommand{\lastx}{0}
\foreach \x [remember=\x as \lastx] in {1, 2, 3, ..., 16} \draw (\lastx) -- (\x);

\foreach \x in {0, 1, 2, 3, ..., 16} \draw [fill=black] (\x)	circle (2pt);
\end{scope}
\end{footnotesize}
\end{tikzpicture}
\caption{A labelled forest on the left, and on the right, its height process on top and its label process below.}
\label{fig:arbre_etiquete}
\end{figure}
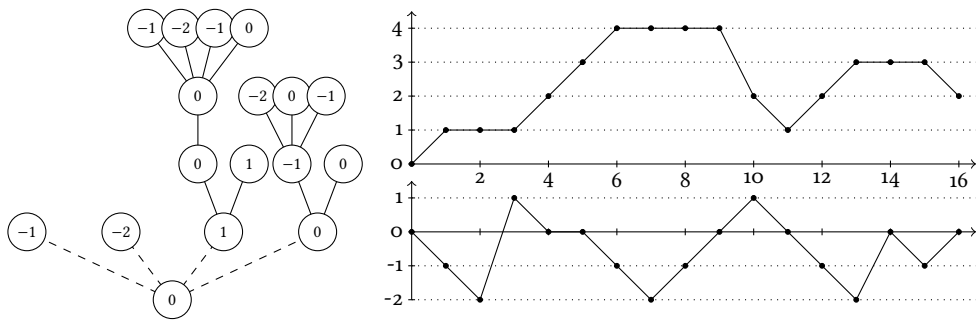

We let $\LFn$ denote the set of forests in $\Fn$ equipped with a labelling as above. Let $\PMn$ be the set of \emph{pointed maps} $(m_n, \star)$ where $m_n$ is a map in $\Mn$ (recall the notation from the introduction) and $\star$ is a distinguished vertex of $\Mn$. If $(m_n, \star)$ is a pointed map, let us denote by $e_+$ and $e_-$ the extremities of the root-edge, so that $\dgr(e_-, \star) = \dgr(e_+, \star)-1$, where $\dgr$ refers to the graph distance; the root-edge is oriented either from $e_-$ to $e_+$, in which case the map is said to be \emph{positive} by Marckert \& Miermont~\cite{Marckert-Miermont:Invariance_principles_for_random_bipartite_planar_maps}, or it is oriented either from $e_+$ to $e_-$, in which case the map is said to be \emph{negative}.

Let us immediately note that since every map in $\Mn$ has $\degr_n(0)$ vertices, then a uniformly random map in $\Mn$ in which we further distinguish a vertex $\star$ uniformly at random has the uniform distribution in $\PMn$. Moreover, half of the $2 \varrho_n$ edges on the boundary are `positively oriented' and half of them are `negatively oriented', so if $\Mn$ is positive, we may re-root it to get a negative map. Therefore it is equivalent to work with random negative maps in $\PMn$ instead of maps in $\Mn$.

This case is simpler to handle; indeed, combining the bijections due to Bouttier, Di Francesco, \& Guitter~\cite{Bouttier-Di_Francesco-Guitter:Planar_maps_as_labeled_mobiles} and to Janson \& Stef\'{a}nsson~\cite{Janson-Stefansson:Scaling_limits_of_random_planar_maps_with_a_unique_large_face}, the set $\LFn$ is in one-to-one correspondence with the set negative maps in $\PMn$. Let us refer to these papers as well as to~\cite{Marzouk:On_scaling_limits_of_planar_maps_with_stable_face_degrees} for a direct construction of the bijection (see also Figure~\ref{fig:bijection_arbre_carte}), and let us only recall here the properties we shall need.
\begin{enumerate}
\item The leaves of the forest are in one-to-one correspondence with the vertices different from the distinguished one in the map, and the label of a leaf minus the infimum over all labels, plus one, equals the graph distance between the corresponding vertex of the map and the distinguished vertex.
\item The internal vertices of the forest are in one to one correspondence with the inner faces of the map, and the out-degree of the vertex is half the degree of the face.
\item The root-face of the map corresponds to the extra root-vertex of the forest, and the out-degree of the latter is half the boundary-length of the map.
\item The number of edges of the map and the forest are equal.
\end{enumerate}
The third property only holds for negative maps, which is the reason why we restricted ourselves to this case.

\begin{figure}[!ht] \centering
\includegraphics[height=7\baselineskip, page = 1]{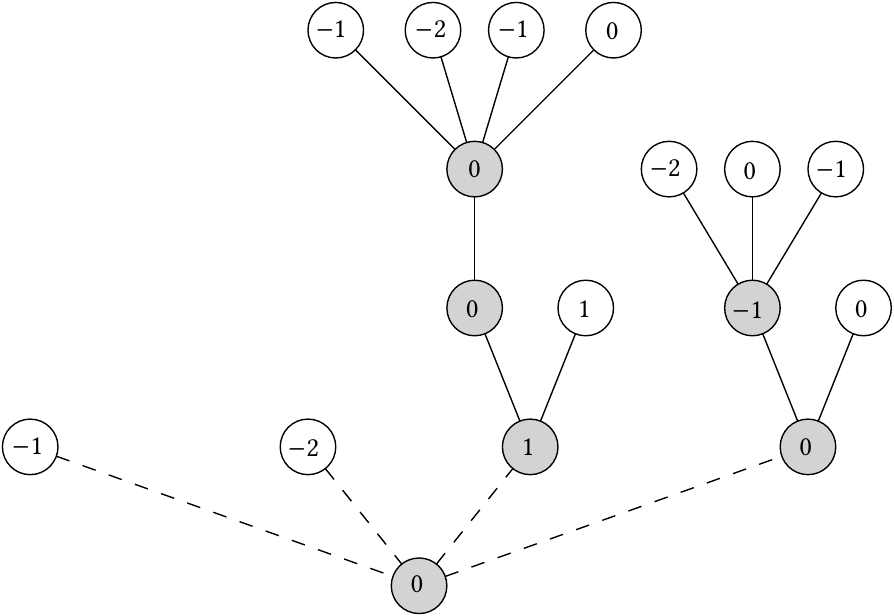}
\qquad
\includegraphics[height=8\baselineskip, page = 2]{Bijection_carte_arbre}
\newline
\includegraphics[height=8\baselineskip, page = 3]{Bijection_carte_arbre}
\caption{The negative pointed map associated with a labelled forest. Labels indicate, up a to a shift, the graph distance to the distinguished vertex, which is the one carrying the smallest label. The map is the one in Figure~\ref{fig:exemple_carte}, seen from a different perspective.}
\label{fig:bijection_arbre_carte}
\end{figure}

\subsection{Tightness of the label process}
\label{sec:tension_etiquettes}

The first key step to understand the asymptotic behaviour of a random map in $\PMn$ is to study the label process of the associated random labelled forest $(\fn, \ell)$ in $\LFn$. The law of the latter can be constructed as follows: first sample $\fn$ uniformly at random in $\Fn$, and then, conditional on it, label its root $0$ and independently for every branch-point with, say, $k \ge 1$ offsprings, sample the label increments between these offsprings and the branch-point uniformly at random in $\mathscr{B}_k^{\ge -1}$.
Let $\Lfn$ be the associated label process. 

\begin{thm}
\label{thm:tension_etiquettes}
Fix any sequence of boundary-lengths $(\varrho_n)_{n \ge 1}$ and any degree sequence $(\degr_n)_{n \ge 1}$ such that $\Delta_n \ge 2$ for every $n \ge 1$ and $\limsup_{n \to \infty} \epsilon_n^{-1} \degr_n(1) < 1$. 
Then from every increasing sequence of integers, one can extract a subsequence along which the label processes
\[\left((\sigma_n + \varrho_n)^{-1/2} \Lfn(\upsilon_n t)\right)_{t \in [0,1]}\]
converge in the space $\mathscr{C}([0,1], \R)$.
\end{thm}

Let us discuss the assumption on $\degr_n(1)$ which is not present in Theorem~\ref{thm:tension_cartes}. The point is that it corresponds to faces of degree $2$ which play no role in the geometry of the map: one could just glue the two edges together for each such face and this would not change the number of vertices nor their distances. So, replacing the degree sequence by $\degr_n(k) \ind{k \ne 1}$, we could assume that $\degr_n(1) = 0$ in fact; we chose a weaker assumption which is easily verified in many interesting models, such as so-called size-conditioned critical stable Boltzmann maps. Such an assumption is needed when coding the map by a labelled forest. 
Indeed, these vertices with out-degree $1$ induce flat steps of the label process so one could think similarly that we could remove them harmlessly, but if $\lim_{n \to \infty} \epsilon_n^{-1} \degr_n(1) = 1$ and $\lim_{n \to \infty} \upsilon_n^{-1} \varrho_n = 0$, then we have in total $\upsilon_n - \degr_n(1) = o(\upsilon_n)$ other vertices so this completely changes the time-scaling and it may induce jumps at the limit (for example if the number of vertices with out-degree $1$ in a small portion of the tree is abnormally small).

Theorem~\ref{thm:tension_etiquettes} extends~\cite[Proposition~7]{Marzouk:Scaling_limits_of_random_bipartite_planar_maps_with_a_prescribed_degree_sequence} which is restricted to the case of a single tree and in the finite-variance regime, when $\sigma_n^2$ is of order $\upsilon_n$. Many arguments generalise here so shall only briefly recall them and focus on the main difference.
First, we shall need a technical result which extends~\cite[Corollary~3]{Marzouk:Scaling_limits_of_random_bipartite_planar_maps_with_a_prescribed_degree_sequence}, a slight adaptation of the proof is needed here.

\begin{lem}\label{lem:bon_evenement_tension_labels}
Recall the notation $\chi_z \in \{1, \dots, k_{pr(z)}\}$ for the relative position of a vertex $z \in \fn$ amongst its siblings. Let
\[c_n = 1 - \frac{\degr_n(0)}{2 \upsilon_n}
\qquad\text{and}\qquad
l_n = \left(\frac{4 \upsilon_n}{\degr_n(0)}\right)^2 \ln\left(\frac{4 \upsilon_n^3}{\degr_n(0)}\right),\]
and consider the event
\[\mathcal{E}_n = 
\bigg\{\frac{\#\{z \in \mathopen{\rrbracket} x, y\rrbracket : \chi_z = 1\}}{\#\mathopen{\rrbracket} x, y\rrbracket}
\le c_n
\text{ whenever } x \text{ is an ancestor of } y \text{ such that } \#\mathopen{\rrbracket} x, y\rrbracket > l_n\bigg\}.\]
Then $\P(\mathcal{E}_n) \to 1$ as $n \to \infty$.
\end{lem}

In words, this means that in $\fn$, there is no branch longer than $l_n$ along which the proportion of individuals which are the left-most (or right-most by symmetry) child of their parent is too large. Let us point out that when $\limsup_{n \to \infty} \epsilon_n^{-1} \degr_n(1) < 1$, then we have seen in Section~\ref{sec:foret_sans_degre_un} that $\liminf_{n \to \infty} \upsilon_n^{-1} \degr_n(0) > 0$, so this length $l_n$ is at most of order $\ln \upsilon_n$ and the proportion $c_n$ is bounded away from $1$.

\begin{proof}
For a vertex $x$ in the forest $\fn$ and $1 \le i \le |x|$, let us denote by $\alpha_i(x)$ the ancestor of $x$ at height $|x|-i+1$, so $\alpha_1(x)$ is $x$ itself, $\alpha_2(x)$ is its parent, etc. Let us write
\[\mathcal{E}_n
= \bigcap_{x \in \fn} \bigcap_{l_n \le l \le |x|} \bigg\{\sum_{i=1}^l \ind{\chi_{\alpha_i(x)} = 1} \le c_n l\bigg\}.\]
Let $x_n$ be a vertex sampled uniformly at random in $\fn$; a union bound yields
\[\Pr{\mathcal{E}_n^c}
\le \upsilon_n \cdot \Pr{\bigcup_{l = l_n}^{|x_n|} \left\{\sum_{i=1}^l \ind{\chi_{\alpha_i(x_n)} = 1} > c_n l\right\}}.\]
Appealing to Lemma~\ref{lem:multi_epines_sans_remise}, since the ratio in~\eqref{eq:epine_un_sommet} is bounded by $1$, this is bounded above by
\[\upsilon_n \sum_{h \le \upsilon_n} \sum_{l = l_n}^h \Pr{\sum_{i=1}^l \ind{\chi_{\degr_n}(i) = 1} > c_n l}.\]
Let us set $g_n = \frac{\upsilon_n - n}{4 \upsilon_n}$, so $c_n = 2g_n + \frac{n}{\upsilon_n}$.
Then a union bound yields
\begin{align*}
\Pr{\sum_{i=1}^l \ind{\chi_{\degr_n}(i) = 1} > c_n l}
&= \Pr{\sum_{i=1}^l \ind{\chi_{\degr_n}(i) = 1} - l \frac{n}{\upsilon_n} > 2g_n l}
\\
&\le \Pr{\sum_{i=1}^l \ind{\chi_{\degr_n}(i) = 1} - \sum_{i=1}^l \frac{1}{\xi_{\degr_n}(i)} > g_nl} 
+ \Pr{\sum_{i=1}^l \frac{1}{\xi_{\degr_n}(i)} - l \frac{n}{\upsilon_n} > g_nl}.
\end{align*}
Note that each $\xi_{\degr_n}(i)^{-1}$ takes values in $[0,1]$ and has mean
\[\Es{\xi_{\degr_n}(i)^{-1}} = \sum_{k \ge 1} k^{-1} \frac{k \degr_n(k)}{\upsilon_n} = \frac{n}{\upsilon_n}.\]
If these variables were obtained by sampling \emph{with} replacement, then we could apply a well-known concentration result, see e.g.~\cite[Theorem~2.3]{McDiarmid:Concentration} to obtain
\[\Pr{\sum_{i=1}^l \xi_{\degr_n}(i)^{-1} - l \frac{n}{\upsilon_n} > g_nl}
\le \exp\left(- 2 g_n^2 l\right).\]
This remains true here since this concentration is obtained by controlling the Laplace transform of $\sum_{i=1}^l \xi_{\degr_n}(i)^{-1}$ and the expectation of any convex function of this sum is bounded above by the corresponding expectation when sampling with replacement, see e.g.~\cite[Proposition~20.6]{Aldous:Saint_Flour}.
Further, conditional on the $\xi_{\degr_n}(i)$'s, the variables $\ind{\chi_{\degr_n}(i) = 1}$ are independent and Bernoulli distributed, with parameter $\xi_{\degr_n}(i)^{-1}$ respectively, so we have similarly
\[\Pr{\sum_{i=1}^l \ind{\chi_{\degr_n}(i) = 1} - \sum_{i=1}^l \xi_{\degr_n}(i)^{-1} > g_nl}
\le \exp\left(- 2 g_n^2 l\right).\]
Since $g_n \in (0,1/2)$, it holds that $1 - \exp(- 2 g_n^2) \ge g_n^2$, so we obtain
\[\Pr{\mathcal{E}_n^c}
\le \upsilon_n \sum_{h \le \upsilon_n} \sum_{l = l_n}^h 2 \exp\left(- 2 g_n^2 l\right)
\le 2 \upsilon_n^2 \frac{\exp(- 2 g_n^2 l_n)}{1 - \exp(- 2 g_n^2)}
\le 2 \upsilon_n^2 g_n^{-2} \exp(- 2 g_n^2 l_n).\]
Finally, we have $l_n 
= g_n^{-2} \ln(g_n^{-1} \upsilon_n^2)$,
so $\P(\mathcal{E}_n^c) 
\le 2 \upsilon_n^{- 2}$,
which converges to $0$.
\end{proof}

We may now prove the tightness of the label process, relying on this result and Proposition~\ref{prop:moments_marche_Luka}.

\begin{proof}[Proof of Theorem~\ref{thm:tension_etiquettes}]
Let $\mathcal{E}_n$ be the event from Lemma~\ref{lem:bon_evenement_tension_labels}, whose probability tends to $1$. We claim that for every $q > 4$, for every $\beta \in (0, q/4-1)$, there exists a constant $C_q > 0$ such that for every $n$ large enough, for every $0 \le s \le t \le 1$, it holds that
\[\Esc{\left|\Lfn(\upsilon_n s) - \Lfn(\upsilon_n t)\right|^q}{\mathcal{E}_n}
\le C_q \cdot (\sigma_n + \varrho_n)^{q/2} \cdot |t - s|^{1+\beta}.\]
The standard Kolmogorov criterion then implies that for every $\gamma \in (0, 1/4)$,
\[\lim_{K \to \infty} \limsup_{n \to \infty} \Prc{\sup_{s \ne t} \frac{|\Lfn(\upsilon_n s) - \Lfn(\upsilon_n t)|}{(\sigma_n + \varrho_n)^{q/2} \cdot |t - s|^\gamma} > K}{\mathcal{E}_n} = 0.\]
Then the same holds for the unconditioned probability and this implies the tightness as claimed.

We may, and do, suppose that $\upsilon_n s$ and $\upsilon_n t$ are integers and that $|t-s| \le 1/2$. Let us view $\fn$ as a tree, let $x$ and $y$ be the vertices visited at time $\upsilon_n s$ and $\upsilon_n t$ respectively, and let $\hat{x}$ and $\hat{y}$ be the children of their last common ancestor which are ancestor of $x$ and $y$ respectively.
Then it was argued in~\cite{Marzouk:Scaling_limits_of_random_bipartite_planar_maps_with_a_prescribed_degree_sequence} that 
\[\Esc{\left|\Lfn(\upsilon_n s) - \Lfn(\upsilon_n t)\right|^q}{\fn}
\le C' \left(\RR(\mathopen{\llbracket} x, \hat{y} \mathclose{\llbracket})^{q/2}
+ \left(\# \mathopen{\rrbracket} \hat{y}, y \mathclose{\rrbracket} + \LL(\mathopen{\rrbracket} \hat{y}, y \mathclose{\rrbracket})\right)^{q/2}\right),\]
where $C' > 0$ is independent of $n$, $s$, and $t$, where $\RR(\mathopen{\llbracket} x, \hat{y} \mathclose{\llbracket})$ counts the number of vertices visited strictly between $x$ and $\hat{y}$ and whose parent is a (strict) ancestor of $x$, and where $\LL(\mathopen{\rrbracket} \hat{y}, y \mathclose{\rrbracket})$ counts the number of vertices visited strictly between $\hat{y}$ and $y$, which are not (strict) ancestor of $y$ but whose parent is. We stress that, as opposed to Section~\ref{sec:Lukasiewicz}, we do view $\fn$ as a single tree here, so if $x$ and $y$ belong to two different components of the forest, then $\hat{x}$ and $\hat{y}$ are their respective roots and $\RR(\mathopen{\llbracket} x, \hat{y} \mathclose{\llbracket})$ takes into account the other roots between them. From Lemma~\ref{lem:codage_marche_Luka}, one gets that $\RR(\mathopen{\llbracket} x, \hat{y} \mathclose{\llbracket}) = \Wfn(\upsilon_n s) - \inf_{r \in [s,t]} \Wfn(\upsilon_n r)$, so, according to Proposition~\ref{prop:moments_marche_Luka}
\[\Es{\RR(\mathopen{\llbracket} u, \hat{v} \mathclose{\llbracket})^{q/2}}
\le C(q/2) \cdot (\sigma_n + \varrho_n)^{q/2} \cdot |t-s|^{q/4},\]

Next, let us consider the branch from $\hat{y}$ to $y$; under the event $\mathcal{E}_n$, if this branch has length greater than $l_n$, then the proportion of individuals which are the first child of their parent is at most $c_n$; all other vertices (a proportion at least $1-c_n$) contribute to $\LL(\mathopen{\rrbracket} \hat{y}, y \mathclose{\rrbracket})$ so there are at most $\LL(\mathopen{\rrbracket} \hat{y}, y \mathclose{\rrbracket})$ of them and therefore
\begin{align*}
\left(\# \mathopen{\rrbracket} \hat{y}, y \mathclose{\rrbracket} + \LL(\mathopen{\rrbracket} \hat{y}, y \mathclose{\rrbracket})\right) \1_{\mathcal{E}_n}
&\le \left(\frac{c_n}{1-c_n} + 1\right) \LL(\mathopen{\rrbracket} \hat{y}, y \mathclose{\rrbracket}) \ind{\mathcal{E}_n \cap \{\# \mathopen{\rrbracket} \hat{y}, y \mathclose{\rrbracket} \ge l_n\}}
+ \left(l_n + \LL(\mathopen{\rrbracket} \hat{y}, y \mathclose{\rrbracket})\right) \ind{\mathcal{E}_n \cap \{\# \mathopen{\rrbracket} \hat{y}, y \mathclose{\rrbracket} < l_n\}}
\\
&\le l_n + \left(\frac{c_n}{1-c_n} + 2\right) \LL(\mathopen{\rrbracket} \hat{y}, y \mathclose{\rrbracket}).
\end{align*}
We have seen in Section~\ref{sec:foret_sans_degre_un} that our assumption that $\degr_n(1) / \epsilon_n$ is bounded away from $1$ implies that $c_n = 1 - \degr_n(0) / (2 \upsilon_n)$ is bounded away from $1$ and so $c_n / (1-c_n)$ is uniformly bounded, say by $K-2$.

Similarly as above, using the `mirror forest' obtained by flipping the order of the children of every vertex, it holds that
\[\Es{\LL(\mathopen{\rrbracket} \hat{v}, v \mathclose{\rrbracket})^{q/2}}
\le C(q/2) \cdot (\sigma_n + \varrho_n)^{q/2} \cdot |t-s|^{q/4}.\]
This yields the bound
\[\Es{\left(\# \mathopen{\rrbracket} \hat{v}, v \mathclose{\rrbracket} + \LL(\mathopen{\rrbracket} \hat{v}, v \mathclose{\rrbracket})\right)^{q/2} \1_{\mathcal{E}_n}}
\le 2^{q/2 - 1} \left(l_n^{q/2} + K^{q/2} C(q/2) \cdot (\sigma_n + \varrho_n)^{q/2} \cdot |t-s|^{q/4}\right).\]
We have seen in Section~\ref{sec:foret_sans_degre_un} that when $\degr_n(1) / \epsilon_n$ is bounded away from $1$, the ratio $\sigma_n^2 / (\upsilon_n - \varrho_n)$ is bounded away from $0$. Then so is $(\upsilon_n)^{-1/2} (\sigma_n + \varrho_n)$. Finally, recall that we have assumed $\upsilon_n s$ and $\upsilon_n t$ to be integers so $|t-s| \ge \upsilon_n^{-1}$, and thus, for $\beta < q/4-1$, it holds that $(\sigma_n + \varrho_n)^{q/2} |t - s|^{1+\beta} \ge (\sigma_n + \varrho_n)^{q/2} \upsilon_n^{-(1+\beta)}$ is bounded below by some positive power of $\upsilon_n$. Since, the threshold $l_n$ is at most of order $\ln \upsilon_n$, we see that for $n$ large enough (but independently of $s$ and $t$), we have $l_n^{q/2} \le (\sigma_n + \varrho_n)^{q/2} |t - s|^{1+\beta}$. It follows that
\begin{align*}
\Es{\left(\# \mathopen{\rrbracket} \hat{v}, v \mathclose{\rrbracket} + \LL(\mathopen{\rrbracket} \hat{v}, v \mathclose{\rrbracket})\right)^{q/2} \1_{\mathcal{E}_n}}
&\le 2^{q/2 - 1} \left((\sigma_n + \varrho_n)^{q/2} |t - s|^{1+\beta} + K^{q/2} C(q/2) \cdot (\sigma_n + \varrho_n)^{q/2} \cdot |t-s|^{q/4}\right)
\\
&\le 2^{q/2 - 1} \left(1 + K^{q/2} C(q/2)\right) \cdot (\sigma_n + \varrho_n)^{q/2} \cdot |t - s|^{1+\beta},
\end{align*}
and the proof is complete.
\end{proof}

\subsection{Tightness of random maps}
\label{sec:tension_cartes}

In this section, we prove tightness of maps as claimed in Theorem~\ref{thm:tension_cartes}, relying on Theorem~\ref{thm:tension_etiquettes} about the label process of the associated labelled forest. The argument finds its root in the work of Le Gall~\cite{Le_Gall:The_topological_structure_of_scaling_limits_of_large_planar_maps} and has now become very standard, adapted e.g. in~\cite{Le_Gall:The_topological_structure_of_scaling_limits_of_large_planar_maps, Le_Gall:Uniqueness_and_universality_of_the_Brownian_map, Abraham:Rescaled_bipartite_planar_maps_converge_to_the_Brownian_map, Bettinelli-Jacob-Miermont:The_scaling_limit_of_uniform_random_plane_maps_via_the_Ambjorn_Budd_bijection, Bettinelli-Miermont:Compact_Brownian_surfaces_I_Brownian_disks, Marzouk:Scaling_limits_of_random_bipartite_planar_maps_with_a_prescribed_degree_sequence, Marzouk:On_scaling_limits_of_planar_maps_with_stable_face_degrees} so we only give the main steps and refer in particular to~\cite[Section~5.3 to~5.5]{Marzouk:On_scaling_limits_of_planar_maps_with_stable_face_degrees} for details.

Let us first briefly define the Gromov--Hausdorff--Prokhorov distance following Miermont~\cite[Proposition~6]{Miermont:Tessellations_of_random_maps_of_arbitrary_genus}, which makes separable and complete the set of measure-preserving isometry classes of compact metric spaces equipped with a Borel probability measure. Let $(X, d_x, m_x)$ and $(Y, d_Y, m_y)$ be two such spaces, a \emph{correspondence} between them is a subset $R \subset X \times Y$ such that for every $x \in X$, there exists $y \in Y$ such that $(x,y) \in R$ and vice-versa. The \emph{distortion} of $R$ is defined as
\[\mathrm{dis}(R) = \sup\left\{\left|d_X(x,x') - d_Y(y,y')\right| ; (x,y), (x', y') \in R\right\}.\]
Then the Gromov--Hausdorff--Prokhorov distance between these spaces is the infimum of all those $\varepsilon > 0$ such that there exists a \emph{coupling} $\nu$ between $m_X$ and $m_Y$ and a compact correspondence $R$ between $X$ and $Y$ such that
\[\nu(R) \ge 1-\varepsilon \qquad\text{and}\qquad \mathrm{dis}(R) \le 2 \varepsilon.\]

\subsubsection{Distances are tight}

Recall that starting from a uniformly random map in $\Mn$, we may sample a vertex $\star$ uniformly at random to obtained a uniformly random pointed map in $\PMn$, which we may then re-root at one of the $\varrho_n$ possible edges on the boundary chosen uniformly at random which make this pointed map $(\mn, \star)$ negative. Let $\mn\setminus\{\star\}$ be the metric measured space given by the vertices of $\mn$ different from $\star$, their graph distance \emph{in $\mn$} and the uniform probability measure and note that the Gromov--Hausdorff--Prokhorov distance between $\mn$ and $\mn\setminus\{\star\}$ is bounded by one so it suffices to prove our claims for $\mn\setminus\{\star\}$.

We want to rely on Theorem~\ref{thm:tension_etiquettes} on the label process of the associated labelled forest $(\fn, \ell)$. This theorem demands the ratio $\degr_n(1) / \epsilon_n$ to be bounded away from $1$. Recall the discussion after the statement of this theorem: we may always if needed discard the faces of degree $2$ by gluing their two edges together; the only risk in doing so is that $\upsilon_n - \degr_n(1)$ does not tend to infinity, in which case Theorem~\ref{thm:tension_etiquettes} does not hold anymore. But if $\upsilon_n - \degr_n(1)$ is bounded, then so are $\degr_n(0)$, $\varrho_n$, and $\sigma_n$ so the number of vertices, faces, and edges of the new map without the faces of degree $2$ is bounded, and so is the normalising factor so tightness is easy in this case. 
In the rest of this section, we therefore assume that $\degr_n(0) \to \infty$ and that we may apply Theorem~\ref{thm:tension_etiquettes} (already extracting a subsequence if necessary).

Recall that in the bijection relating $(\mn, \star)$ to a labelled forest $(\fn, \ell)$, the vertices of the former different from $\star$ correspond to the leaves of the latter, and the internal vertices of $\fn$ are identified with their last child. Therefore, for every vertex $x$ of $\fn$ (viewed as a tree), we let $\varphi(x)$ be the vertex of $\mn\setminus\{\star\}$ in one-to-one correspondence with the leaf at the extremity of the right-most ancestral line starting from $x$ in $\fn$. Let us list the vertices of $\fn$ as $x_0 < x_1 < \dots < x_{\upsilon_n}$ in lexicographical order and for every $i,j \in \{0, \dots, \upsilon_n\}$, we set
\[d_n(i,j) = \dgr(\varphi(x_i), \varphi(x_j)),\]
where $\dgr$ is the graph distance in $\mn$. We then extend $d_n$ to a continuous function on $[0, \upsilon_n]^2$ by `bilinear interpolation' on each square of the form $[i,i+1] \times [j,j+1]$ as in~\cite[Section~2.5]{Le_Gall:Uniqueness_and_universality_of_the_Brownian_map} or~\cite[Section~7]{Le_Gall-Miermont:Scaling_limits_of_random_planar_maps_with_large_faces}. For every $0 \le s \le t \le 1$, let us set
\[d_{(n)}(s, t) = (\sigma_n + \varrho_n)^{-1/2} d_n(\upsilon_n s, \upsilon_n t)
\qquad\text{and}\qquad
L_{(n)}(t) = (\sigma_n + \varrho_n)^{-1/2} \Lfn(\upsilon_n t).\]

For a continuous function $g : [0, 1] \to \R$, let us set for every $0 \le s \le t \le 1$,
\[D_g(s,t) = D_g(t,s) = g(s) + g(t) - 2 \max\left\{\min_{r \in [s, t]} g(r); \min_{r \in [0, s] \cup [t, 1]} g(r)\right\}.\]
In a slightly different context (the coding of~\cite{Bouttier-Di_Francesco-Guitter:Planar_maps_as_labeled_mobiles}, but it is adapted in this context, see~\cite{Marzouk:Scaling_limits_of_random_bipartite_planar_maps_with_a_prescribed_degree_sequence}) Le Gall~\cite[Equation~4]{Le_Gall:Uniqueness_and_universality_of_the_Brownian_map} (see also~\cite[Lemma~3.1]{Le_Gall:The_topological_structure_of_scaling_limits_of_large_planar_maps} for a detailed proof) obtained the bound
\[d_n(\upsilon_n s, \upsilon_n t) \le D_{\Lfn}(\upsilon_n s, \upsilon_n t) + 2\]
for every $s,t \in [0,1]$. 
According to Theorem~\ref{thm:tension_etiquettes}, from every increasing sequence of integers, one can extract a subsequence along which the continuous processes $L_{(n)}$ converge in distribution to some limit process, say $L$. From the previous bound, one can extract a further subsequence along which we have
\begin{equation}\label{eq:convergence_distances_sous_suite}
\left(L_{(n)}(t), D_{L_{(n)}}(s, t), d_{(n)}(s, t)\right)_{s,t \in [0,1]}
\cvloi
(L_t, D_L(s,t), d_\infty(s,t))_{s,t \in [0,1]},
\end{equation}
where $(d_\infty(s,t))_{s,t \in [0,1]}$ depends a priori on the subsequence and satisfies $d_\infty \le D_L$, see~\cite[Proposition~3.2]{Le_Gall:The_topological_structure_of_scaling_limits_of_large_planar_maps} for a detailed proof in a similar context.

\subsubsection{Planar maps are tight}

Let us implicitly restrict ourselves to a subsequence along which~\eqref{eq:convergence_distances_sous_suite} holds.
Appealing to Skorokhod's representation theorem, let us assume furthermore that it holds almost surely. The fonction $d_\infty$ is continuous on $[0,1]^2$ and is a pseudo-distance (as limit of $d_{(n)}$). We then define an equivalence relation on $[0,1]$ by setting
\[s \approx t
\qquad\text{if and only if}\qquad
d_\infty(s,t) = 0,\]
and we let $M_\infty$ be the quotient $[0,1] / \approx$, equipped with the metric induced by $d_\infty$, which we still denote by $d_\infty$. We let $\Pi$ be the canonical projection from $[0,1]$ to $M_\infty$ which is continuous (since $d_\infty$ is) so $(M_\infty, d_\infty)$ is a compact metric space, which finally we endow with the Borel probability measure $m_\infty$ given by the push-forward by $\Pi$ of the Lebesgue measure on $[0,1]$. We claim that, deterministically, from~\eqref{eq:convergence_distances_sous_suite}, we may deduce the convergence
\[\left(\Vmn\setminus\{\star\}, (\sigma_n + \varrho_n)^{-1/2} \dgr, \pgr\right) \cv (M_\infty, d_\infty, m_\infty)\]
for the Gromov--Hausdorff--Prokhorov distance, this will end the proof of Theorem~\ref{thm:tension_cartes}.

The sequence $(\varphi(x_i))_{0 \le i \le \upsilon_n}$ lists \emph{with redundancies} the vertices of $\mn$ different from $\star$. For every $1 \le i \le \degr_n(0)$, let us denote by $\lambda(i) \in \{1, \dots, \upsilon_n\}$ the index such that $x_{\lambda(i)}$ is the $i$-th leaf of $\tn$, and extend $\lambda$ linearly between integer times. The function $\lambda$ corresponds to $\zeta_{\{-1\}}(\,\cdot\,; \Wfn)$ with the notation of Lemma~\ref{lem:proportion_feuilles}. According to this lemma, we have
\begin{equation}\label{eq:approximation_sites_aretes_carte}
\left(\upsilon_n^{-1} \lambda(\lceil\degr_n(0) t\rceil) ; t \in [0,1]\right) \cvproba (t ; t \in [0,1]),
\end{equation}
Observe that the sequence $(\varphi(x_{\lambda(i)}))_{1 \le i \le \degr_n(0)}$ now lists \emph{without redundancies} the vertices of $\mn$ different from $\star$. The set
\[\mathscr{R}_n = \left\{\left(\varphi(x_{\lambda(\lceil \degr_n(0) t\rceil)}), \Pi(t)\right) ; t \in [0,1]\right\}.\]
is a correspondence between $\mn\setminus\{\star\}$ and $M_\infty$. Let further $\nu$ be the coupling between $\pgr$ and $m_\infty$ given by
\[\int_{(\mn\setminus\{\star\}) \times M_\infty} f(x, z) \d\nu(x, z) = \int_0^1 f\left(\varphi(x_{\lambda(\lceil \degr_n(0) t\rceil)}), \Pi(t)\right) \d t,\]
for every test function $f$. Then $\nu$ is supported by $\mathscr{R}_n$ by construction. Finally, the distortion of $\mathscr{R}_n$ is given by
\[\sup_{s,t \in [0,1]} \left|d_{(n)}\left(\frac{\lambda(\lceil \degr_n(0) s\rceil)}{\upsilon_n}, \frac{\lambda(\lceil \degr_n(0) t\rceil)}{\upsilon_n}\right) - d_\infty(s,t)\right|,\]
which, appealing to~\eqref{eq:approximation_sites_aretes_carte}, tends to $0$ whenever the convergence~\eqref{eq:convergence_distances_sous_suite} holds, which concludes the proof of Theorem~\ref{thm:tension_cartes}.


{\small
\linespread{1}\selectfont

}

\end{document}